\documentclass[10pt, a4paper]{article}
\usepackage[utf8]{inputenc}
\usepackage{amsmath}
\usepackage{mathtools}
\usepackage{mathrsfs} 
\usepackage{enumitem}
\usepackage{amsfonts}
\usepackage{amssymb}
\usepackage{amsthm}
\usepackage{indentfirst}
\usepackage{graphicx}
\usepackage{xcolor}
\usepackage{color, soul}

\usepackage{bookmark}

\usepackage[left=35mm,right=35mm,top=35mm,bottom=35mm,a4paper]{geometry}

\newcommand{\footremember}[2]{%
    \footnote{#2}
    \newcounter{#1}
    \setcounter{#1}{\value{footnote}}%
}
\newcommand{\footrecall}[1]{%
    \footnotemark[\value{#1}]%
} 

\author{%
  Jan Giesselmann\footremember{Jan}{Department of Mathematics, Technical University of Darmstadt, Dolivostr. 15, 64293 Darmstadt, Germany}
  \and Kiwoong Kwon\footremember{knu}{Department of Mathematics, Kyungpook National University, Daegu, Republic of Korea}%
  \and Min-Gi Lee\footrecall{knu}\footremember{MG}{Corresponding author; email: leem@knu.ac.kr\newline\newline Acknowledgement: We extend our gratitude to Professor Athanasios Tzavaras for generously sharing his valuable insights.}%
  }

\usepackage{stackengine}
\stackMath

\def\Orb{{\textsf{Orbit}}\,}

\def\E{{\mathcal{E}}}
\def\H{{\mathcal{H}}}
\def\K{{\mathcal{K}}}

\def\Lip{{\textsf{Lip}}}
\def\drr{{d\rho_0(x)d\rho_0(x')}}
\def\dr{{d\rho_0(x)}}
\def\bdot{{\boldsymbol{\cdot}}}

\def\tr{\,\textrm{tr}\,}
\def\div{\,\textrm{div}\,}
\def\sgn{\,\textrm{sgn}\,}

\def\eps{\varepsilon}

\newcommand{\mres}{\mathbin{\vrule height 1.6ex depth 0pt width
0.13ex\vrule height 0.13ex depth 0pt width 1.3ex}}

\newtheorem{theorem}{Theorem}
\theoremstyle{definition}
\newtheorem{definition}[subsubsection]{Definition}

\newtheorem{appendixlemma}[subsection]{Lemma}
\newtheorem{corollary}[subsubsection]{Corollary}
\newtheorem{proposition}[subsubsection]{Proposition}
\theoremstyle{remark}
\newtheorem{remark}[subsubsection]{Remark}

\date{}

\title{Relative entropy technique in terms of position and momentum and its application to Euler-Poisson system}

\begin{document}

\maketitle

\abstract
This paper presents a systematic study of the relative entropy technique for compressible motions of continuum bodies described  as  Hamiltonian flows. While the description for the classical mechanics of $N$ particles involves a Hamiltonian in terms of position and momentum vectors, that for the continuum fluid involves a Hamiltonian in terms of density and momentum. For space dimension $d\ge 2$, the Hamiltonian functional has a non-convex dependency on the deformation gradient or placement map due to material frame indifference. Because of this, the applicability of the relative entropy technique with respect to the deformation gradient or the placement map is inherently limited. Despite these limitations, we delineate the feasible applications and limitations of the technique by pushing it to its available extent. Specifically, we derive the relative Hamiltonian identity, where the Hamiltonian takes the position and momentum field as its primary and conjugate state variables, all within the context of the referential coordinate system that describes the motion.

This approach, when applicable, turns out to yield rather strong stability statements.
As instances, we consider Euler-Poisson systems in one space dimension. For a specific pressureless model, we verify non-increasing \(L^2\) state differences before the formation of \(\delta\)-shock. In addition, weak-strong uniqueness, stability of rarefaction waves, and convergence to the gradient flow in the singular limit of large friction are shown. Depending on the presence or absence of pressure, assumptions are made to suitably accommodate phenomena such as $\delta$-shocks, vacuums, and shock discontinuities in the weak solutions.
\tableofcontents

\section{Introduction}
This article studies the relative entropy technique for   compressible motions of  continuum bodies modeled as  Hamiltonian flows. We consider a continuum body as a collection of matter, and the motion of the body is described by the placement map $\eta(t,\cdot)$\footnote{We follow the terminology in \cite{1991Truesdell}.}, mapping from the reference coordinate system $X$ to the spatial coordinate system $Y$. The placement map $\eta(t,\cdot)$ is a state variable, and the momentum field $m(t,\cdot)$ is its conjugate variable. The Hamiltonian $\H$ is a functional on the pair $(\eta(t,\cdot),m(t,\cdot))$, and the Hamiltonian flow  is represented by  
\begin{equation} \label{HamIntro} \tag{$H$}
    \dot{\eta}=\frac{\delta \H}{\delta {m}}, \quad \dot{{m}}=-\frac{\delta \H}{\delta \eta}.
\end{equation}
The applications of the relative entropy method in hyperbolic problems can be traced back to the works \cite{1979DiPerna} and \cite{1979Dafermos,1979Dafermosa}. This method, rooted in the framework of the second law of thermodynamics, has been widely employed in the examination of stability of solutions in hyperbolic systems of conservation laws with convex entropy. It has been used to prove the weak-strong uniqueness, to conduct singular limit analysis for model convergence, and to obtain stability-related estimates.

The main task of this article is to examine the following quantity
\begin{equation} \label{introrelham}
    \begin{aligned}
    &\H[\eta,m] - \H[\bar\eta,\bar m] - \Big\langle\frac{\delta \H}{\delta \eta}[\bar\eta,\bar m], (\eta-\bar\eta)\Big\rangle- \Big\langle\frac{\delta \H}{\delta m}[\bar\eta,\bar m], (m-\bar m)\Big\rangle\\
     &=:\H[\eta,m \,|\,\bar\eta, \bar m],
    \end{aligned}
\end{equation}
as a {\it relative entropy} in the use of the relative entropy technique. 
To avoid any confusion, the term entropy is reserved exclusively for instances where the relative entropy method itself is invoked. We adhere to the term relative Hamiltonian for the quantity represented by \eqref{introrelham}. The relative Hamiltonian \eqref{introrelham} is between two solutions $(\eta,m)$ and $(\bar\eta,\bar m)$ of equation \eqref{HamIntro}. A formula for the time derivative of $\H[\eta,m \,|\,\bar\eta, \bar m]$, that is the \textit{relative Hamiltonian identity}, will be derived. We primarily focus on Hamiltonian functionals from isentropic Euler and Euler-Poisson equations.
We consider relative Hamiltonians with respect to the pair $(\eta(t,\cdot),m(t,\cdot))$, consisting of our primary and conjugate state variables.

Despite differences in origin between the classical relative entropy method, which is developed based on the second law of thermodynamics, and the above quantity, which is based on a conservative Hamiltonian, the application of the technique leads, at least formally, to the surprisingly simple formula of the relative Hamiltonian identity \eqref{dotrel}.

To better illustrate our purpose, we briefly introduce three well-established instances of relative entropies and their consequences. In the first instance, studied in \cite{2016Giesselmann}, we consider systems of equations of the form:
\begin{equation} \label{intro:Euler}
    \begin{aligned}
    \frac{\partial \rho}{\partial t}+\operatorname{div}(\rho u) &=0 \\ 
    \rho\left(\frac{\partial u    }{\partial t}+(u \cdot \nabla) u\right) &=-\rho\nabla \frac{\delta{\mathcal{E}}}{\delta \rho}(\rho)
    \end{aligned} \quad \quad y \in \mathbb{R}^d, \: t>0,
\end{equation}
where the functional
$$ \H[\rho,n]:= \K[\rho,n] + \E[\rho] = \int_{\mathbb{R}^d} \frac{|n|^2}{2\rho} \:dy  + \mathcal{E}[\rho], \quad n = \rho u, $$
is formally conserved. $\E[\rho]$ is a given interaction energy functional. The system is presented in a spatial description, with $y$ denoting the spatial coordinate. The Isentropic Euler, Euler-Poisson, and Euler-Korteweg equations are included in this framework. In this spatial description, the time evolution of the following quantity is monitored:
\begin{align*}
    \H[\rho,n\,|\,\bar\rho, \bar n]&:= \K[\rho,n\,|\,\bar\rho, \bar n] + \E[\rho\,|\,\bar\rho]\\
    &= \int_{\mathbb{R}^d} \frac{\rho(u-\bar u)^2}{2} \:dy + \E[\rho] - \E[\bar\rho] - \Big\langle\frac{\delta \E}{\delta \rho}[\bar\rho], (\rho-\bar \rho)\Big\rangle.
\end{align*}
Here,  the relative Hamiltonian functional is defined with respect to the state variables $(\rho,n)$. Notably, $\rho(t,y)$ and $n(t,y)$ represent the mass density and the momentum of the infinitesimal matter occupying the spatial coordinate $y$ at time $t$. In \cite{2016Giesselmann}, the authors demonstrated that, for two smooth solutions, the time derivative of the above quantity is formally given by:
\begin{equation} \label{intro:Edotrel}
    \begin{aligned}
     &\frac{d}{dt} \left( \E[\rho|\bar\rho] + \int_{\mathbb{R}^d} \frac{1}{2}\rho|u-\bar u|^2 \:dy\right) \\ &+\int_{\mathbb{R}^d} \rho \nabla_y \bar{u}:(u-\bar{u})\otimes(u-\bar{u}) \:dy - \int_{\mathbb{R}^d} \nabla_y \bar{u} : S(\rho|\bar\rho)\:dy = 0.
    \end{aligned}
\end{equation}
Here, $\rho \mapsto S(\rho)$ denotes the stress operator, defined in accordance with $\E[\rho]$, i.e. satisfying $\operatorname{div}(S(\rho)) = \rho \nabla\frac{\delta \E}{\delta \rho}[\bar\rho]$. The {\it relative stress} is given by $S(\rho|\bar\rho):= S(\rho) - S(\bar\rho) - \left\langle \frac{\delta S}{\delta \rho}(\bar\rho), \rho-\bar\rho\right\rangle$. In practice, one of the solutions $(\rho,n)$ is allowed to be a suitable energy-decreasing weak solution, and \eqref{intro:Edotrel} becomes an inequality. The last two integrals of \eqref{intro:Edotrel} are bounded by the relative Hamiltonian itself, leading to a stability estimate via Gronwall's inequality. Aside from the outcomes of \eqref{intro:Edotrel}, such as weak-strong uniqueness, asymptotic model convergences, and stability estimates, it is important to note that these stability studies are conducted in a formally unified framework, regardless of the specific form of $\E[\rho]$.

The second instance is the classical mechanics of $N$ particles. The motion can be described by Hamilton's ordinary differential equations. Given the Hamiltonian 
\begin{align*}
    \H(q_1,q_2, \cdots,q_N, p_1,p_2, \cdots,p_n)&:= \K(p_1,p_2,\cdots,p_N) + \mathcal{U}(q_1,q_2, \cdots,q_N) \\
    &=\sum_{i=1}^N \frac{|p_i|^2}{2M_i} + \mathcal{U}(q_1,q_2, \cdots,q_N),    
\end{align*}
where for $i=1,\cdots,N$, $(q_i,p_i)$ is a pair of position and momentum vectors and $M_i$ is the mass of the $i$-th particle,
$$ \dot{q}_i = \frac{\partial \H}{\partial p_i}, \quad \dot{p}_i = -\frac{\partial \H}{\partial q_i}, \quad i=1,\cdots,N.$$
This description is referential, indexed discretely by $i=1,\cdots,N$. The quantity of interest is
\begin{equation} \label{intro:Npart}
\begin{aligned}
    &\H[(q_i),(p_i)\,|\, (\bar{q}_i), (\bar{p}_i)]:=\H[(q_i),(p_i)] - \H[(\bar{q}_i), (\bar{p}_i)] \\
    &- \sum_{k=1}^N \partial_{q_k} \H[(\bar{q}_i),(\bar{p}_i)](q_k - \bar{q}_k) - \sum_{k=1}^N \partial_{p_k} \H[(\bar{q}_i),(\bar{p}_i)](p_k - \bar{p}_k)\\
    &= \sum_{i=1}^N \frac{|p_i-\bar{p}_i|^2}{2M_i} + \mathcal{U}[(q_i)] - \mathcal{U}[(\bar{q}_i)] - \sum_{k=1}^N \partial_{q_k} \mathcal{U}[(\bar{q}_i)](q_k - \bar{q}_k)\\
    &= \K[(p_i)\,|\,(\bar{p}_i)] + \mathcal{U}[(q_i)\,|\,(\bar{q}_i)].
\end{aligned}
\end{equation}
Here, the relative Hamiltonian is defined with respect to the state variables $\big( p_i, q_i\big)$ of position and momentum vectors. It is a calculus exercise to derive the time derivative formula of the above quantity. The time derivative formula is given by
\begin{equation} \label{intro:Npartid}
    \begin{aligned}
        \frac{d}{dt}&\Big(\K[(p_i)\,|\,(\bar{p}_i)] + \mathcal{U}[(q_i)\,|\,(\bar{q}_i)] \Big) \\
        &+ \sum_{k=1}^N \dot{\bar{q}}_k \bigg(-\partial_{q_k} \mathcal{U}[({q}_i)] + \partial_{q_k} \mathcal{U}[(\bar{q}_i)] + \sum_{\ell=1}^N \partial_{q_\ell}\partial_{q_k} \mathcal{U}[(\bar{q}_i)](q_\ell - \bar{q}_\ell) \bigg) = 0. 
    \end{aligned}
\end{equation}
Considering the physical dimensions of terms appearing in \eqref{intro:Npartid}, we call the quantity in the parenthesis in the summand the $k$-th {\it relative force} and the summand the $k$-th {\it relative work rate}. Equation \eqref{intro:Npartid} expresses how the relative Hamiltonian changes in time.

In the last instance, we turn our attention to the referential description of polyconvex elastodynamics, a model for the elastic deformation of a solid body. In this model, the elastic energy is given by 
$$ \E[F] =  \int_X W(F) \:dx, $$
where $X$ is the referential coordinate system, $F(t,\cdot)=D\eta(t,\cdot)$ denotes the deformation gradient, and $W$ stands for the stored energy density. If the stored energy density $W$ depends on $F$ through a convex function $g$ of minors of $F$, then the elastic energy is called polyconvex. For example, in the three-dimensional case ($d=3$), the stored energy density function can be written as $W(F) = g\circ \Psi(F)$, where $\Psi(F) = (F,\textrm{cof}\, F, \textrm{det}\, F)$. 

One specific formulation of polyconvex elastodynamics was introduced in \cite{1998Qin}, \cite{2001Demoulini} in such a way that the resulting system of equations becomes a hyperbolic system of conservation laws with a convex entropy. This formulation treats all minors of $F$ as state variables, and the time evolution equations of minors are appended to the system.   See \cite[Section 5.4 contingent entropies and polyconvexity]{2016Dafermos} for further details.
In the three-dimensional case, the quantity of interest is
\begin{align*}
    &\H[F,\textrm{cof}\, F, \textrm{det}\, F,m] - \H[\bar F,\textrm{cof}\, \bar F, \textrm{det}\, \bar F,\bar m] - \Big\langle\frac{\delta\H}{\delta m}[\bar F,\textrm{cof}\, \bar F, \textrm{det}\, \bar F, \bar m], m-\bar m \Big\rangle\\
    &- \Big\langle\frac{\delta\H}{\delta F}[\bar F,\textrm{cof}\, \bar F, \textrm{det}\, \bar F, \bar m], F-\bar F \Big\rangle - \Big\langle\frac{\delta\H}{\delta \textrm{cof}\, F}[\bar F,\textrm{cof}\, \bar F, \textrm{det}\, \bar F, \bar m], \textrm{cof}\,F-\textrm{cof}\, \bar F \Big\rangle \\
    &- \Big\langle\frac{\delta\H}{\delta \textrm{det}\, F}[\bar F,\textrm{cof}\, \bar F, \textrm{det}\, \bar F, \bar m], \textrm{det}\,F-\textrm{det}\, \bar F \Big\rangle.
\end{align*}

In the second and third instances, the descriptions are referential. Consequently, the state variables at a given time $t$ monitor the infinitesimal matter at a fixed reference point, denoted respectively by the index $i$ and the coordinate $x$.

Having considered the three instances previously discussed, we now shift our attention to our core issue. It appears that the implications of utilizing the relative Hamiltonian \eqref{introrelham} have remained unexplored, despite it being a direct generalization of \eqref{intro:Npart}. Specifically, the Hamiltonian takes the placement map $\eta(t,\cdot)$ as the primary state variable, and the momentum field $m(t,\cdot)$ as its conjugate variable, all within the context of the referential coordinate system that describes the motion.

The topic of our intended study has not been addressed before. This is, at least partially, due to the subtleties that arise when dealing with dimensions $d\ge2$. As observed in the first instance concerning fluid dynamics and the third instance involving polyconvex elastodynamics, in space dimensions $d\ge2$, the Hamiltonian functional depends on the deformation gradient $D\eta$ through the determinant $\textrm{det}\, D\eta$ (refer to the formal relationship between Eulerian density $\rho$ and $D\eta$ in Section \ref{EPderivation}) or other minors. Furthermore, the principle of material frame indifference necessitates a non-convex dependence, see \cite{1959Coleman}, \cite{1976Ball}, and \cite{2001Demoulini}. Consequently, for the majority of continuum mechanics models in space dimension $d\ge 2$,  the Hamiltonian functional does neither depend convexly on the state variables $D\eta$ nor on $\eta$. As a result, we lose the bridge from the relative entropy identity to the stability estimates, since the leading order expansion of the Hamiltonian up to the first order will not be positive definite.

Nevertheless, we think it is valuable to clarify the relative energy identity formula in terms of the state variable $(\eta(t,\cdot),m(t,\cdot))$. This is because confirming the identity and making use of it (in the presence of strict convexity) for stability analysis are two independent tasks.

In this regard, our first objective is to write the system \eqref{HamIntro} in referential description which formally implies the Euler-like fluid system in spatial description. Subsequently, we aim to formally derive a unified relative Hamiltonian identity formula \eqref{dotrel} that resembles \eqref{intro:Npartid}. Once this is accomplished, if the Hamiltonian depends on the state variables in a strictly convex manner, we can then obtain stability estimates. As a specific example, for which this is the case, we consider the Euler-Poisson system in one space dimension.

Despite the restricted applicability, the method, when applicable, yields stability statements that are quite strong. Among the rigorous results we present, one outcome involves a particular pressureless Euler-Poisson system in Section \ref{specialEP}. For two solutions of this system, we compute the $L^2$ difference, which remains non-increasing until the possible formation of a $\delta$-shock: 
$$\int_X \frac{|v-\bar v|^2}{2} + \frac{|\eta - \bar\eta_{a(t)}|^2}{2} \:\dr\bigg|_{t} \le \int_X \frac{|v-\bar v|^2}{2} + \frac{|\eta - \bar\eta_{a(0)}|^2}{2} \:\dr\bigg|_{t=0} \quad \text{for a.e. $t$.}$$
In this equation, $v$ is the velocity field, where $m=\rho_0 v$, and $\bar\eta_{a(t)}(t,\cdot)$ denotes the placement map, shifted by the constant $a(t)$, to ensure that the centers of mass of $\eta$ and $\bar\eta$ coincide. Even after the possible formation of a $\delta$-shock, the $L^2$ distance grows only at a rate of $t^{1/2}$. This yields a bound that is considerably stronger than what is typically obtainable from Gronwall's inequality.

The formulation of specific problems, subject to particular sets of regularity assumptions, is detailed in Section 3. We provide a rigorous justification for all calculations in subsequent stability estimates concerning Euler-Poisson systems in one spatial dimension, which are consequences of aforementioned specific regularity assumptions. It is important to have assumptions that accommodate physically meaningful weak solutions, such as shock waves. Given that pressureless models can exhibit severer singularities, we present two separate formulations: one for models with pressure and another for models without pressure. In Section 3.1, readers will find a review of the literature on the singularity formation of such systems. We also note that our stability results rely on the relative entropy method, which allows to compare two solutions, namely a weak solution and a strong solution, i.e. one of the solutions requires higher regularity. Therefore, our formulation is designed in such a way that the sufficiently permissive assumptions accommodating singularities apply only to the weak solution.

In cases where pressure is present, weak solutions are defined within the regularity class of bi-Lipschitz placement maps (see Section \ref{sec:(A)}). It is known that when pressure is present the formation of a $\delta$-shock, i.e., the creation of an atomic part in $\rho$ is prevented. The solution may still exhibit shock discontinuities and vacuums. Whether a vacuum can be generated from smooth initial data remains unknown, to the best of our knowledge. In this first formulation, since a $\delta$-shock is not anticipated, our formulation permits the weak solution to exhibit shock discontinuity but excludes the possibility of $\delta$-shock. Unfortunately, we are unable to address the vacuum case in the bi-Lipschitz class. 

In cases where pressure is absent, weak solutions are defined within the regularity class of monotone placement maps (see Section \ref{sec:(B)}). This formulation allows the formation of both vacuum and $\delta$-shocks. When this happens, the Eulerian density $\rho$ becomes singular in the sense that it has atomic parts or it touches $0$, respectively. In terms of properties of the map $\eta(t,\cdot)$, these cases correspond to a nontrivial interval over which the map $\eta(t,\cdot)$ is constant and a point of jump discontinuity of $\eta(t,\cdot)$, respectively. In particular, the map $\eta(t,\cdot)$ is always a monotone function of locally bounded variation. In particular, the set of monotone increasing placement maps is a cone. In regard to this, in the study of \cite{2013Brenier} the weak solution $\eta(t,\cdot)$ is assumed to incorporate an additional imposed acceleration, forcing the map to lie in the cone if the map $\eta(t,\cdot)$ reaches the boundary of the cone. We adopt their ideas in our formulation.

As previously mentioned, our method is limited to the comparison of the weak solution with a strong solution. For models with pressure, the strong solution cannot exhibit any discontinuities. In the pressureless model, both $\delta$-shocks and vacuums are prohibited for the strong solution. See Sections \ref{sec:(A)} and \ref{sec:(B)}.

The organization of this paper is as follows: In Section 2, we consider a Hamiltonian  \(\mathcal{H}\) such that \eqref{HamIntro} becomes the Euler-Poisson system. We formally demonstrate that the referential description \eqref{HamIntro} indeed implies the Euler-Poisson system. Importantly in this section we present the formal derivation of the relative Hamiltonian identity.

In Section 3 and subsequent sections, we consider Euler-Poisson systems in one space dimension as an example where the Hamiltonian  depends strictly convexly manner on the state variables. In Section 3, we provide a brief introduction to the problem and collect materials that are necessary for the following sections. In Section 4, we present stability estimates for a wide range of Euler-Poisson type systems, with or without pressure. We provide proofs for weak-strong uniqueness and the stability of rarefaction waves. In Section 5, we study a specific pressureless Euler-Poisson model, where the sharpest results are obtained. The results presented in this section take into account the possibility of vacuum and $\delta$-shocks in  the weak solution. We verify non-increasing \(L^2\) state differences before the \(\delta\)-shock forms. Furthermore, we prove  convergence to the gradient flow in the singular limit of large friction.

\section{Formal relative entropy identity of Hamiltonian flow} \label{sec:2}

The Euler equations for an isentropic compressible gas is an archetype of the problems we will discuss in this article. The system can be formally derived from the energy functional, called the {\it Hamiltonian}, given by
\begin{equation} \label{baro}\H[\rho, {n}] = \int_{\mathbb{R}^d} \frac{1}{2}\frac{| {n}|^2}{\rho}  + \rho^\gamma \:dy, \quad \gamma>1. \end{equation}

For a class of models describing a compressible motion of a body, being a Hamiltonian system certainly is an additional property besides being a  system of hyperbolic conservation laws. One of the most explicit ways to take advantage of this special property is the relative entropy technique. 
See the extensive work by \cite{2016Giesselmann}, where the relative entropy technique is applied to a wide range of purposes. We remark that one of the applications in \cite{2016Giesselmann} includes the Euler-Korteweg system with a non-convex pressure law.

Our goal is to investigate what happens if we use the placement map $\eta(t,\cdot)$ as one of the state variables. The relationship between the placement map $\eta(t,\cdot)$ and the Eulerian density $\rho$ is specified in \eqref{nomen}. Other than this, we follow the same process as in the classical relative entropy technique. We directly write with $\H = \H[\eta,m]$ the Hamiltonian system
\begin{equation} \label{Ham}
    \dot{\eta}=\frac{\delta \H}{\delta {m}}, \quad \dot{{m}}=-\frac{\delta \H}{\delta \eta}\,.
\end{equation}
 In the following section, we establish that \eqref{Ham} formally implies the Euler equations. 

This spirit of research, seeking comprehension by looking at problems from  both the placement map $\eta(t,x)$ perspective and the Eulerian density $\rho(t,y)$ perspective, has been vital in the study of gradient flows and inspires our work. We reference the celebrated works of \cite{2001Otto}, \cite{2008Ambrosio}.  

\subsection{Euler-Poisson system as a Hamiltonian flow} \label{EPderivation}
As a case for our study, we consider the Hamiltonian of the Euler-Poisson system. We establish that the equations \eqref{Ham}
formally imply the Euler-Poisson system. In order to connect Lagrangian and Eulerian descriptions, we provide a summary of the notations and terminology commonly used in continuum mechanics. We acknowledge that the computations made in Section \ref{sec:2} are formal. As an instance of this, the map $\eta(t,\cdot)$ is assumed to be smooth and bijective for each $t$, with a smooth inverse. This inverse is denoted by $\eta^{-1}(t,\cdot)$.

\noindent
\begin{flalign} \label{nomen}
    \textbf{Nomenclature : } &&
\end{flalign}
{\allowdisplaybreaks
\begin{align*}
 \quad& X = Y = \mathbb{R}^d, \\
 x \in X \quad&: \text{a point in referential or Lagrangian coordinate system},\\
 y \in Y \quad&: \text{a point in spatial or Eulerian coordinate system},\\
 \eta: (t,x) \mapsto y \quad&: \text{a motion map,}\\
 F^i_\alpha(t,x) = \frac{\partial \eta^i}{\partial x^\alpha}(t,x)\quad&: \text{$(i,\alpha)$-entry of deformation gradient matrix $F$,}\\
 J(t,x) = \det F(t,x) \quad&: \text{Jacobian determinant,}\\
 \rho_0(x) \quad&: \text{mass density to the Lebesgue measure of $X$,}\\
 \rho(t,y) \quad&: \text{mass density push-forwarded by $\eta(t,\cdot)$,}\\	
 \tau(t,x) = \frac{J(t,x)}{\rho_0(x)} \quad &: \text{specific volume,}\\
 {v}(t,x)=\dot{\eta}(t,x)\quad&: \text{velocity vector,}\\
 {u}(t,y)=v(t,\eta^{-1}(t,y))\quad&: \text{Eulerian velocity vector,}\\
 {m}(t,x)=\rho_0(x) {v}(t,x)\quad&: \text{momentum density,}\\ 
 n(t,y)=\rho(t,y)u(t,y) \quad&: \text{Eulerian momentum density.}
\end{align*}
The quantity $\frac{\partial \eta^i}{\partial x^\alpha}$ is referred to as a deformation gradient, a term commonly used in continuum mechanics, though here the map $\eta(t,\cdot)$ is a placement, not a deformation.
}

The Euler-Poisson system (in Eulerian description) is given as follows:
\begin{equation} \label{EP}
    \begin{aligned}\frac{\partial \rho}{\partial t}+\operatorname{div}(\rho u) &=0 \\\rho\left(\frac{\partial u}{\partial t}+(u \cdot \nabla) u\right) &=-\nabla p(\rho)-\rho \nabla c \\-\Delta c+\beta c &=\rho-(\rho),
    \end{aligned}
\end{equation}
where $\beta\ge 0$, and $(\rho)=\int_{\mathbb{R}^d} \rho\:dy$.

For this particular system, the Hamiltonian is represented as\footnotemark[2]:
\begin{equation}\label{EPham}
\begin{aligned}
 \tilde\H[\rho,n] &= \tilde\K[\rho,n] + \tilde\E[\rho], \quad (n=\rho u),\\
 \tilde\K[\rho,n] & = \int_Y \frac{1}{2}\frac{|n|^2}{\rho} \: dy,\\
 \tilde\E[\rho]&=\int_Y h(\rho) \:dy +\frac{1}{2} \iint_{Y\times Y} \rho(z) {K}(y-z) \rho(y) \:d z d y = \tilde\E_1[\rho] + \tilde\E_2[\rho]
\end{aligned}
\end{equation}
for a given $h(\cdot)$ and the kernel $K$ such that $c=K*\rho$ solves the Poisson equation $\eqref{EP}_3$. Especially, when $\beta = 0$, we may choose ${K}(y)=\Phi(y)+\frac{|y|^{2}}{2d}$, where $\Phi(y)$ is the fundamental solution of the Laplace's equation, i.e., $-\Delta \Phi = \delta(y)$. \footnotetext[2]{One might consider including $c$ within the state variables, writing 
$$\hat{\E}[\rho,c] = \int h(\rho) \:dy + \frac{1}{2} \int \rho c \: dy.$$}

The first step involves expressing the integrals in \eqref{EPham} within the referential coordinate system. Using the map $\eta(t,\cdot)$ we have
\begin{equation} \label{EPham2}
\begin{aligned}
 \tilde\K[\rho,n] &= \int_X \frac{|m|^2}{2\rho_0(x)} \:dx=:\K[m],\\
 \tilde\E_1[\rho] &= \int_X \psi(\tau) \rho_0(x)\:dx=:\E_1[\eta], \quad \psi(s) = sh\left(\frac{1}{s}\right),\\
 \tilde\E_2[\rho] &= \frac{1}{2}\iint_{X\times X} \rho_0(x) K\big(\eta(t,x) - \eta(t,x')\big) \rho_0(x') \: dxdx'=:\E_2[\eta].
\end{aligned}
\end{equation}
Upon comparing \eqref{EPham2} with \eqref{EPham}, we note that $\K=\K[m]$ and $\E=\E[\eta]$. 
This implies $$\frac{\delta\H}{\delta m} = \frac{\delta\K}{\delta m} \quad \text{and} \quad \frac{\delta\H}{\delta \eta} = \frac{\delta\E}{\delta \eta}.$$

In the second step, we formally compute the variations as follows:
\begin{align*}
 \left\langle\frac{\delta \K}{\delta m}, \varphi\right\rangle = \frac{d}{ds} \K[m +s\varphi] \Big|_{s=0} \quad \text{and} \quad  \left\langle\frac{\delta \E}{\delta \eta}, \phi\right\rangle = \frac{d}{ds} \E[\eta +s\phi] \Big|_{s=0}
\end{align*}
where $\varphi$ and $\phi$ are in $C_c^\infty(X)$.
It is straightforward to verify that Equation $\eqref{Ham}_1$ implies:
$$\dot{\eta} = \frac{m}{\rho_0} = v.$$

We are left to investigate the implication of Equation $\eqref{Ham}_2$. To evaluate $\left\langle-\frac{\delta \E_1}{\delta \eta},\phi\right\rangle$, we calculate the following:
$$-\psi^\prime(\tau) \lim_{\epsilon \rightarrow 0} \frac{1}{\epsilon} \big(\tau[\eta+\epsilon\phi] - \tau[\eta]\big).$$ 
In what follows, let $[A^i_\alpha]$ represent the matrix whose entries are $A^i_\alpha$ for $i,\alpha= 1,\cdots, d$. In addition, for a function $f$ on $X$, we define $\tilde f(t,y):=f(\eta^{-1}(t,y))$. We have:
\begin{align*}
    \det \left[\frac{\partial(\eta^i+\epsilon\phi^i)}{\partial x^\alpha}\right] &= \det\left[\frac{\partial \eta^i}{\partial x^\alpha} + \epsilon\frac{\partial\phi^i}{\partial x^\alpha}\right] = \det\left[\frac{\partial \eta^i}{\partial x^\alpha} + \epsilon\sum_j \frac{\partial\tilde\phi^i}{\partial y^j}\bigg|_{y=\eta(t,x)}\frac{\partial \eta^j}{\partial x^\alpha}\right]\\
    &= \det \left[\Big(\textrm{I} + \epsilon \nabla_y \tilde\phi\big|_{y=\eta(t,x)}\Big) \left[\frac{\partial \eta^{i}}{\partial x^\alpha}\right]\right]\\ 
    &=\Big(1 + \epsilon \tr \:\nabla_y \tilde\phi\big|_{y=\eta(t,x)}\Big) \det \left[\frac{\partial \eta^{i}}{\partial x^\alpha}\right] + \text{higher order terms}.
   \end{align*}
Thus, we have
\begin{align*}
 &\lim_{\epsilon \rightarrow 0} \frac{1}{\epsilon} \Big( \det \left[\frac{\partial(\eta^i+\epsilon\phi^i)}{\partial x^\alpha}\right] - \det \left[\frac{\partial\eta^i}{\partial x^\alpha}\right]\Big) = \det \left[\frac{\partial \eta^i}{\partial x^\alpha}\right] \div_y\tilde\phi\bigg|_{y=\eta(t,x)}, \\
 \left\langle-\frac{\delta\E_1}{\delta\eta},\phi\right\rangle &= -\int_X \psi^\prime(\tau)\,\tau \div_y \tilde\phi\big|_{y=\eta(t,x)} \,\rho_0(x) \: dx \\
 &= -\int_X \psi^\prime(\tau)\,\div_y \tilde\phi\big|_{y=\eta(t,x)} \,J(t,x) \: dx.
\end{align*}
Note that $\psi'(\tau) = h\left(\frac{1}{\tau}\right) - \frac{1}{\tau}h'\left(\frac{1}{\tau}\right) = h(\rho) - \rho h'(\rho)$, and the last integral can be written as
\begin{align}
 \int_Y \big(\rho h'(\rho)- h(\rho)\big) \div_y \tilde\phi \: dy 
 = -\int_Y \nabla_y\big(\rho h'(\rho)- h(\rho)\big) \tilde\phi \: dy. \nonumber
\end{align}
We write $\rho h^\prime(\rho) - h(\rho)=:p(\rho)$. Nextly, 
\begin{equation} \nonumber
    \begin{aligned}
        \left\langle-\frac{\delta \E_2}{\delta \eta}, \phi\right\rangle &=-\left.\frac{d}{d \epsilon}\right|_{\epsilon=0} \frac{1}{2} \iint_{X\times X} \rho_{0}(x) {K}(\eta(x)+\epsilon \phi(x)-\eta(x')-\epsilon \phi(x')) \rho_{0}(x') d x d x' \\
        &=-\frac{1}{2} \iint_{X\times X} \rho_{0}(x) D {K}(\eta(x)-\eta(x')) \cdot(\phi(x)-\phi(x')) \rho_{0}(x') d x d x' \\
        &=-\frac{1}{2} \int_X \rho_{0}(x) \phi(x) \cdot \int_X D {K}(\eta(x)-\eta(x')) \rho_{0}(x') d x' d x \\
        &+\frac{1}{2} \int_X \rho_{0}(x') \phi(x') \cdot \int_X D {K}(\eta(x)-\eta(x')) \rho_{0}(x) d x d x'.
        \end{aligned}
\end{equation}
If we can assume
\begin{equation} \nonumber
    {K}(y) \text { is an even function of } y \text { and } D {K}(y) \text { is an odd function of } y,
\end{equation}
which is the case for the radially symmetric kernel, 
then 
\begin{equation*}
    \left\langle-\frac{\delta \E_{2}}{\delta \eta}, \phi\right\rangle=-\int_X \rho_{0}(x) \phi(x) \cdot \int_X D{K}(\eta(x)-\eta(x')) \rho_{0}(x') d x' d x.
\end{equation*}
The double integral can be written as 
\begin{equation} \nonumber
    \begin{aligned}
        &-\int_Y \rho(y) \tilde\phi(y) \cdot \int_Y D {K}(y-y') \rho(y') d y d y'\\
        &=-\int _Y\rho(y) \tilde\phi(y) \cdot(D{K} * \rho)(y) d y \\
        &=-\int_Y \rho(y) \tilde\phi(y) \cdot \nabla_{y}c \:d y.
        \end{aligned}
\end{equation}
In conclusion,
\begin{align*}
\left\langle -\frac{\delta \H}{\delta \eta}, \phi \right\rangle &= - \int_Y \Big(\nabla_y p(\rho) + \rho\nabla_y c\Big)\tilde\phi \: dy.
\end{align*}
Since 
\begin{align*}
\int_X \dot{m}(t,x) \phi(x)\:dx &= \int_Y \rho \Big(\tilde v_t + (\tilde v \cdot \nabla_y ) \tilde v\Big)  \tilde\phi \:dy
\end{align*}
we recover the momentum equation $\eqref{EP}_2$. The Poisson equation $\eqref{EP}_3$ has been implemented by defining $c= K*\rho$. The derivation of the continuity equation $\eqref{EP}_1$ from the compatibility equation $\dot{F^i_\alpha} = \partial_\alpha v^i$, where $i,\alpha=1,\cdots, d$, is omitted.

\subsection{Formal relative Hamiltonian identity}
Having established the formal implications of \eqref{Ham}, we proceed to formally calculate the relative Hamiltonian identity. The Hamiltonian is given in the form 
\begin{equation*}
H[\eta, {m}]=\int_{X} \frac{|{m}|^{2}}{2\rho_{0}} d x+\mathcal{E}[\eta].
\end{equation*}

Let $(\eta, {m})$ and $(\bar{\eta}, \bar{{m}})$ be two smooth solutions of \eqref{Ham}. We define the relative Hamiltonian as
\begin{equation*} 
\H[\eta, {m} \mid \bar{\eta}, \bar{{m}}]=\H[\eta, {m}]-\H[\bar{\eta}, \bar{{m}}]-\left\langle\frac{\delta \H}{\delta \eta}(\bar{\eta}, \bar{{m}}), \eta-\bar{\eta}\right\rangle-\left\langle\frac{\delta \H}{\delta {m}}(\bar{\eta}, \bar{{m}}), {m}-\bar{{m}}\right\rangle.
\end{equation*}
We calculate the time derivative of the relative Hamiltonian. For simplicity, we write $\H=\H[\eta, {m}]$, $\bar{\H} = \H[\bar{\eta}, \bar{{m}}]$ and denote $\E$, $\bar\E$, $\K$, and $\bar\K$ in a similar manner. Considering that
$$\H[\eta,m] = \K[m] + \E[\eta] \quad\Longrightarrow \quad \frac{\delta \K}{\delta \eta} = \frac{\delta \E}{\delta m} = \frac{\delta^2 \H}{\delta \eta \delta m} = 0,$$
we have:
\begin{equation*}
    \begin{aligned}
        \frac{d}{dt}\H[\eta, m | \bar{\eta}, \bar{m}] = &~\frac{\delta \K}{\delta {m}}\dot{{m}} - \frac{\delta \bar{\K}}{\delta {m}}\dot{\bar{{m}}} - \frac{\delta^2 \bar{\K}}{\delta {m}^2} \dot{\bar{{m}}}({m}-\bar{m}) - \frac{\delta \bar{\K}}{\delta {m}}(\dot{{m}} - \dot{\bar{{m}}}) \\
        + &~\frac{\delta \mathcal{\E}}{\delta \eta}~\dot\eta~ - \frac{\delta \mathcal{\bar{\E}}}{\delta \eta}~\dot{\bar{\eta}}~ - \frac{\delta ^2 \bar{\mathcal{\E}}}{\delta \eta ^2}~\dot{\bar{\eta}}~(\eta ~-~ \bar{\eta}) - \frac{\delta \bar{\mathcal{\E}}}{\delta \eta}(\dot{\eta} ~-~ \dot{\bar{\eta}}) \\
        =&~\dot{\bar{{m}}}\left(\frac{\delta \K}{\delta m} - \frac{\delta \bar{\K}}{\delta m} - \frac{\delta^2 \bar{\K}}{\delta m^2}({m} - \bar{{m}})\right) \\
        +&~\dot{\bar{\eta}}~\left(\frac{\delta \mathcal{\E}}{\delta \eta}~ ~- \frac{\delta \bar{\mathcal{\E}}}{\delta \eta} - \frac{\delta^2 \bar{\mathcal{\E}}}{\delta \eta^2}(\eta ~-~ \bar{\eta})\right) \\
        +&~ \frac{\delta \K}{\delta m} \dot{{m}} - \frac{\delta {\K}}{\delta m}\dot{\bar{{m}}}  - \frac{\delta \bar{\K}}{\delta {m}}\dot{{m}}+ \frac{\delta\bar{\K}}{\delta {m}}\dot{\bar{{m}}} \\
        +&~\frac{\delta \mathcal{\E}}{\delta \eta} ~\dot{\eta}~ - \frac{\delta {\mathcal{\E}}}{\delta \eta}~\dot{\bar{\eta}}~  - \frac{\delta \bar{\mathcal{\E}}}{\delta \eta}~\dot{\eta}~+ \frac{\delta\bar{\mathcal{\E}}}{\delta \eta}~\dot{\bar{\eta}}~.
    \end{aligned}
\end{equation*}
Because of the equations \eqref{Ham}, 
\begin{equation*}
    \frac{\delta K}{\delta m} \dot{{m}} - \frac{\delta {K}}{\delta m}\dot{\bar{{m}}}  - \frac{\delta \bar{K}}{\delta {m}}\dot{{m}}+ \frac{\delta\bar{K}}{\delta {m}}\dot{\bar{{m}}} + \frac{\delta \mathcal{E}}{\delta \eta} \dot{{\eta}} - \frac{\delta {\mathcal{E}}}{\delta \eta}\dot{\bar{{\eta}}}  - \frac{\delta \bar{\mathcal{E}}}{\delta {\eta}}\dot{{\eta}}+ \frac{\delta\bar{\mathcal{E}}}{\delta {\eta}}\dot{\bar{{\eta}}} = 0.
\end{equation*}
Because $\K[m]$ is quadratic in $m$ ($\frac{\delta \K}{\delta m} = v$ and $\frac{\delta^2 \K}{\delta^2 m} = \frac{1}{\rho_0}$), 
$$\frac{\delta K}{\delta m} - \frac{\delta \bar{K}}{\delta m} - \frac{\delta^2 \bar{K}}{\delta m^2}({m} - \bar{{m}})=0.$$
Therefore the identity simplifies to 
\begin{equation} \label{dotrel}
    \frac{d}{d t} \H[\eta, {m} \mid \bar{\eta}, \bar{{m}}]+\left\langle\bar{{v}}, -\frac{\delta \E}{\delta \eta}+\frac{\delta \bar{\E}}{\delta \eta}+\frac{\delta^{2} \bar{\E}}{\delta \eta^{2}}(\eta-\bar{\eta})\right\rangle=0,
\end{equation}
where the formal interpretation is that $\frac{\delta \E}{\delta \eta}$ and $\frac{\delta \bar\E}{\delta \eta}$ are applied to $\bar{{v}}$, and $\frac{\delta^2 \bar\E}{\delta^2 \eta}$ is applied to $\bar{{v}}$ and $(\eta-\bar{\eta})$. 

We define  the abbreviations
$$ \mathbf{f}[\eta|\bar\eta]:=-\frac{\delta \E}{\delta \eta}+\frac{\delta \bar{\E}}{\delta \eta}+\frac{\delta^{2} \bar{\E}}{\delta \eta^{2}}(\eta-\bar{\eta}), \quad \left\langle\bar{{v}}, \mathbf{f}[\eta|\bar\eta]\right\rangle$$
and refer to them as {\it relative force} functional and the {\it relative work rate}, respectively. 

The formal identity \eqref{dotrel} is evidently simple and results in several remarkable estimates, which we will discuss later. The mathematical structure of \eqref{dotrel} bears similarities to the relative entropy results obtained for the Euler equations in \cite{1979DiPerna,1979Dafermosa,1979Dafermos}. However, there are two crucial differences. The first is that the results in \cite{1979DiPerna,1979Dafermosa,1979Dafermos} rely on the thermodynamic structure of the models, which is induced by the Gibbs-Duhem inequality, while our results stem from a Hamiltonian structure of the models, similar to the relative energy estimates in \cite{2016Giesselmann,2023Egger}. The second, and more apparent, difference lies in the choice of state variables; those used in this work differ from those in previous studies; note also the difference in state variables between \cite{2016Giesselmann} and \cite{2023Egger}.

To better understand the second difference, we offer two observations comparing \eqref{dotrel} and \eqref{intro:Edotrel}:

\begin{enumerate}
 \item The two relative kinetic energies $$\int_Y \frac{1}{2} \rho|u(y)-\bar u(y)|^2\:dy \quad \text{and} \quad \int_X \frac{1}{2} \rho_0|v(x)-\bar{v}(x)|^2 \:dx$$
are not identical. The integrand of the former measures the velocity differences observed at a specific spatial point $y$, whereas the latter measures the differences between the velocities of particles that share a common label $x$. These two particles may not be at the same spatial location. More specifically, the latter integral can be expressed in spatial coordinates using either $\eta$ or $\bar\eta$:
 $$ \int_Y \frac{1}{2} \rho\big|v\circ\eta^{-1}(y) - \bar{v}\circ\eta^{-1}(y)\big|^2 \:dy  = \int_Y \frac{1}{2} \bar\rho \big|v\circ\bar\eta^{-1}(y) - \bar{v}\circ\bar\eta^{-1}(y)\big|^2 \:dy.$$
 \item The convective 
 second term in \eqref{intro:Edotrel} is absent in \eqref{dotrel}. This absence is analogous to the change in the momentum balance where the convection term $ \div (\rho u \otimes u)$ is absorbed into the material time derivative. 
\end{enumerate}

In the following Sections \ref{sec:EPknown} to \ref{specialEP}, we present an example where \eqref{dotrel} delivers meaningful results. We investigate the Euler-Poisson systems in one space dimension, with all calculations justified.

\section{Interacting systems of Euler-Poisson type} \label{sec:EPknown}
The typical system under consideration is as follows:
\begin{equation} \label{Esys}
\begin{aligned}
\frac{\partial \rho}{\partial t}+\operatorname{div}(\rho u) &=0, \\
\rho\left(\frac{\partial u}{\partial t}+(u \cdot \nabla) u\right) &=-\nabla p(\rho)-\rho \nabla c, \\
c & = K*\rho,
\end{aligned} \quad \quad y \in \mathbb{R}^d, \quad t>0
\end{equation}
Here, we consider a given kernel $K$ and the pressure law $\rho \mapsto p(\rho)$. Throughout our investigation, we only consider radially symmetric kernels. Equation \eqref{Esys} is commonly referred to as the Euler-Poisson system, particularly when $c$ involves the inversion of a type of Poisson equation, such as
$$-\Delta c + \beta c = \sigma\big(\rho - (\rho)\big),$$
where $\beta \ge 0$, $\sigma \in \{-1,1\}$, along with some boundary condition when the domain is not the whole space. When the domain is the whole space and the kernel $K(y)= \sigma c_\alpha|y|^{-\alpha}$ for a suitable range of $\alpha$ and a constant $c_\alpha>0$, the system \eqref{Esys} is also often referred to as the Euler-Riesz system \cite{2021Choia,2022Choi}. 

From the perspective of interaction energy, and more specifically the gradient flows, cases where the kernel $K$ is radially increasing are referred to as \textit{attractive} cases, while those where the kernel $K$ is radially decreasing are referred to as \textit{repulsive} cases. For more details, please refer to the discussion in Section \ref{sec:convexity}.

In the following, we provide a brief overview of some background and previous results on continuum mechanical systems that involve interactions of a non-local and pressure type. Two different sets of assumptions are presented, which we  call the formulation (A) and (B) below. They are used in Section \ref{sec:EP} and Section \ref{specialEP}, respectively.

\subsection{Background and previous results}
Results on the Euler-Poisson system in $d \geq 2$ dimensions remain limited. The challenges involved in multiple space dimensions are discussed in \cite[Section 5]{2012Bae}.
The formation of finite-time singularities was analyzed for attractive kernel \cite{2008Chae}. For the repulsive kernels,  radial solutions have been investigated by \cite{2012Bae, 2023Rozanova}. Their studies reveal that the majority of radial solutions for the multi-dimensional Euler-Poisson system blow up. On the other hand, the study \cite{1998Guo} showed that the linearized Euler-Poisson system in three space dimension has a global-in-time solution if it has small irrotational initial data. More recent work on the Euler-Riesz system \cite{2022Choi} established the local well-posedness of smooth solutions and the formation of finite-time singularities.

Understanding of the systems in one spatial dimension has significantly developed. Unless otherwise specified, the reported results henceforth  pertain to the one-dimensional problem.

Mechanical systems without pressure have attracted substantial attention, e.g. \cite{2015Cavalletti}. In these systems, there is usually no mechanism preventing the formation of a $\delta$-shock, i.e., the creation of an atomic part in $\rho$. Particularly, there is extensive research on systems where motion is purely driven by inertia, and particles stick together inelastically upon collision and remain aggregated afterwards. 
This model is known as the sticky particle system. Numerous studies on the sticky particle system can be found in \cite{1996Weinan,1998Brenier,2009Gangbo,2009Natile} and references therein.

The authors of \cite{2013Brenier} conducted a comprehensive study on the sticky particle system with interactions. They formulated the problem as differential inclusions instead of equations similar to \eqref{HamIntro}. This was primarily because the map $\eta(t,\cdot)$ could incorporate an imposed acceleration, forcing the map to lie in the cone of monotone increasing maps, if the map $\eta(t,\cdot)$ reaches the boundary of the cone. Using the theory of differential inclusion, they demonstrated the existence of a Lagrangian solution, the stability of solutions, and the convergence of finite particle approximation schemes for a wide range of interaction types. Additional insights can be found in the following references \cite{2020Hynd, 2022Choi}.

Now, we examine previous findings on the Euler-Poisson system (the system with non-local interactions and pressure) in one spatial dimension. 
The Euler-Poisson system with pressure does not allow mass concentration in the form of $\delta$-shocks \cite{2008Tadmor}:
As long as the total mass is finite, it is shown in \cite{2008Tadmor} that both the velocity $u$ and the density $\rho$ remain uniformly bounded. However, $u_y$, the derivative of the velocity field, can experience a finite-time blow-up, resulting in a shock discontinuity. The threshold phenomena that distinguish between two potential outcomes, namely, finite-time shock discontinuity formation and global existence of a smooth solution, have been thoroughly examined   in \cite{1996Engelberg,2001Engelberg,2002Liu,2008Tadmor}. Surprisingly, the criteria for these threshold phenomena for specific models have been accurately quantified in terms of initial velocity and density. We do not attempt to provide an exhaustive list of the research conducted on the Euler-Poisson system. 
Instead, we refer readers to the key references \cite{2008Chae, 2016Guo}
and the references cited therein.

Aside from singularities like shock discontinuity and $\delta$-shock, the system may also exhibit vacuum states. Solutions with a vacuum exist for the isentropic Euler systems too: Self-similar solutions with a vacuum arising from Riemann initial data can be found in \cite{2000Chen}, and the existence theory proposed by \cite{1983DiPerna,1996Lions,2000Chena} relies on the compensated compactness approach, which does not treat the vacuum as  a special case. Alternatively, \cite{2009Jang,2015Jang} tackled the problem in multiple spatial dimensions and utilized a method that traces the vacuum boundary. They demonstrated the local well-posedness of smooth solutions with a vacuum. For a more comprehensive understanding of the vacuum problem, we direct readers to \cite{2011Jang}. 

The work presented in the following Sections \ref{sec:EP} and \ref{specialEP} is formulated based on previously established results. 
%

\subsection{Placement and motion maps}

Let $X=Y=\mathbb{R}$ serve as the referential and spatial coordinate systems respectively. The placement map $\eta: X \mapsto Y$ accounts for the spatial location of matter referenced by $x\in X$. We define a class of placement maps 
\begin{align*}
 {\textsf{S}}(X)&:= \left\{\eta : X \rightarrow Y ~|~ \text{$\eta$ is increasing}\right\}
\intertext{and a class of motion maps over a time interval $I=[a,b)$}
 {{\Orb}}(I\times X)&:= \left\{\eta : I\times X \rightarrow Y ~|~  \text{for all $t \in I$, \:$\eta(t,\cdot) \in {\textsf{S}}(X)$} \: \right\}.
\end{align*}

We currently do not consider any topology for ${\textsf{S}}(X)$. One can verify that ${\textsf{S}}(X)$ is a cone.
We fix a measure $\rho_0$ on $X$, representing the mass distribution in the referential frame. Suppose $\E : {\textsf{S}}(X) \rightarrow \mathbb{R}$ is a given interaction energy functional. It is typical that $\E[\eta]$ is an integral with respect to the measure $\rho_0$. Given these, we say a map $\eta \in {\textsf{S}}(X)$ has finite interaction energy if $\E[\eta]$ is bounded.
If a motion $\eta \in {{\Orb}}(I\times X)$ is such that for each $t\in I$, $\dot\eta(t,\cdot)\in L^1_{\rho_0}(X) \cap L^2_{\rho_0}(X)$, then we say $\eta$ has finite momentum and kinetic energy. Here, $\dot\eta(t,x)$ denotes $\partial_t\eta(t,x)$.

We also define a class of bi-Lipschitz placement maps on $X$
\begin{align*}
 \hat{\textsf{S}}(X)&:= \left\{\eta \in \textsf{S}(X) ~|~ 
\text{$\eta$ is injective, surjective, strictly increasing,} \right.\\
&~~~~~~~~~~~~~~~~~~~~\left.\text{$\eta$ and $\eta^{-1}$ are Lipschitz.}\:\right\} \\
\intertext{and a class of motion maps over a time segment interval $I=[a,b)$}
 \hat{{\Orb}}(I\times X)&:= \left\{\eta : I\times X \rightarrow Y ~|~  \text{for all $t \in I$, \:$\eta(t,\cdot) \in \hat{\textsf{S}}(X)$} \: \right\}.
\end{align*}

\subsection{Formulation (A) for Section \ref{sec:EP}} \label{sec:(A)}
Here we collect assumptions used in Section \ref{sec:EP}. We take $\rho_0$ as the Lebesgue measure restricted to an open subset $\tilde{X}$ of $X$, i.e.,
\begin{equation} \nonumber
 \rho_0 = \mathcal{L}\mres\tilde{X}.
\end{equation}
Having defined $\rho_0$, we consider an energy functional on $\hat{\textsf{S}}(X)$ of the form
\begin{equation} 
 \E = \E_{r} + \E_{a} + \E_{w}, \label{E}
\end{equation}
where each term, $\E_{r}$, $\E_{a}$, and $\E_{w}$, is defined as follows: The functionals $\E_r$ and $\E_a$ are non-local interaction energies of double integral form with respect to the measure $\rho_0 \times \rho_0$. We first fix an $\eta_* \in \hat{\textsf{S}}(X)$, and
\begin{equation} \label{ErEa}
    \begin{aligned}
        \E_r[\eta] &=\iint \Big({K}_r\big(\eta(x)-\eta(x')\big) -  {{K}_r\big(\eta_*(x)-\eta_*(x')\big)}\Big) d\rho_{0}(x)d\rho_0(x'), \\ 
        \E_a[\eta] &=\iint \Big({K}_a\big(\eta(x)-\eta(x')\big) -  {{K}_a\big(\eta_*(x)-\eta_*(x')\big)}\Big) d\rho_{0}(x)d\rho_0(x'),
    \end{aligned}
\end{equation}
where the kernels $K_r$ and $K_a$ are of power law type,
\begin{equation} \label{kernellaw}
    {K}_r(y) = -\frac{|y|^p}{p} \quad(-1< p \le 1, p\ne 0) \quad \text{ and }\quad {K}_a(y) = \frac{|y|^q}{q} \quad(q \ge 1).
\end{equation}
The choices of ${K}_r$ and ${K}_a$ ensure that they are symmetric and convex in the radial variable, with ${K}_r$ radially decreasing and ${K}_a$ radially increasing. We will provide further explanations for these choices in the following section. The rationale behind the subtraction of $K_r\big(\eta_*(x)-\eta_*(x')\big)$ and $K_a\big(\eta_*(x)-\eta_*(x')\big)$ in the integrands will also be provided.
Next, we define the energy that gives rise to the pressure. We adopt one appropriate for isentropic ideal gas. For any given $\eta\in \hat{\textsf{S}}(X)$, we take $\tau = \frac{\eta_x}{\rho_0}\ge 0$ as a $\rho_0$-a.e. well-defined function, which we call the specific volume. With this notation and for $\gamma>1$, we define $\E_{w}$ as
\begin{equation}\label{Ep}
\begin{aligned} 
 \E_{w}[\eta] &= \int_X \left(\left(\frac{\eta_x}{\rho_0}\right)^{1-\gamma}  - \left(\frac{\eta_{*x}}{\rho_0}\right)^{1-\gamma}\right)\: d\rho_0(x)\\
 &= \int_X \Big(\psi(\tau) - \psi(\tau_*) \Big) \: \dr, \quad \psi(s) = s^{1-\gamma}.
\end{aligned}
\end{equation}
Finally, we define the class of admissible maps $\hat{\mathcal{A}}(X)$ and motions $\hat{\mathcal{A}}(I\times X)$ as follows. A placement map $\eta \in \hat{\textsf{S}}(X)$ is \textit{admissible} if 
\begin{enumerate} [label=(\roman*)]
 \item $\Big|K_r\big(\eta(x) -\eta(x')\big)-K_r\big(\eta_*(x) -\eta_*(x')\big)\Big|$ and {$\allowbreak \Big|K_a\big(\eta(x) -\eta(x')\big)-K_a\big(\eta_*(x) -\eta_*(x')\big)\Big|$ }are integrable with respect to the measure $\rho_0\times \rho_0$,
 \item $\left|\left(\frac{\eta_x}{\rho_0}\right)^{1-\gamma}  - \left(\frac{\eta_{*x}}{\rho_0}\right)^{1-\gamma}\right|$ is integrable with repect to the measure $\rho_0$,
\end{enumerate}
and we write $\eta \in \hat{\mathcal{A}}(X)$. We define $\hat{\mathcal{A}}(I\times X)$ as the  collection of $\eta \in \hat{\Orb}(I\times X)$ for which $\eta(t,\cdot)$ is admissible, and $\dot\eta(t,\cdot) \in L^1_{\rho_0}(X)\cap L^2_{\rho_0}(X)$ for all $t\in I$.  

\subsection{Formulation (B) for Section \ref{specialEP}} \label{sec:(B)}
In this second formulation, $\E_{w}$ is dropped and for $\E_r$ and $\E_a$, $p=1$ and $q=2$ are chosen. Advantages of writing the relative energy with the placement map $\eta(t,\cdot)$ as state variable is best illustrated for this special presureless system. In this formulation we assume
$$\rho_0 \in L^1(X) \quad \text{and  $\int_X \dr = 1$.}$$
The interaction energy is normalized as
\begin{align} \label{energynormal}
 \E[\eta] = \frac{1}{4}\iint_{X\times X} -|\eta(x) - \eta(x')|+ |\eta(x) - \eta(x')|^2 \: \drr,
\end{align}
and we define
$$ K(y) := K_r(y) + K_a(y), \quad K_r(y) = -\frac{|y|}{4}, \quad K_a(y) = \frac{y^2}{4}.$$

To see the energy \eqref{energynormal} as that of the pressureless Euler-Poisson system in one space dimension (see Section \ref{EPderivation}), we take
$$ \hat{K}(y) = \hat{K}_r(y) + \hat{K}_a(y), \quad \hat{K}_r(y) = - \frac{|y|}{2}, \quad \hat{K}_a(y) = \frac{y^2}{2}.$$
This leads to 
\begin{align*}
&\E[\eta] = \frac{1}{2}\iint_{X\times X} \hat{K}\big(\eta(x)-\eta(x')\big) \:\drr, \\
&-\partial^2_{yy} \hat{K}_r(y) = \delta_0, \quad -\partial^2_{yy} \hat{K}_a(y) \equiv -1 
\end{align*}
and $c$ defined by 
\begin{align*}
 c(y) = \int_\mathbb{R} \hat{K}(y-z) \rho(z)\: dz
\end{align*}
solves the Poisson equation
\begin{align*}
 -\partial^2_{yy} c(y) = \rho(y) - (\rho), \quad (\rho):=\int_\mathbb{R} \rho(y) \:dy.
\end{align*}
We will mainly use  $K_r(y) = -\frac{|y|}{4}$ and $K_a(y) = \frac{y^2}{4}$ in what follows.

Based on this interaction energy we define the class of admissible motions ${\mathcal{A}}(I\times X)$ as the collection of $\eta \in {\Orb}(I\times X)$ satisfying 
\begin{equation} \label{assumption}
    \begin{aligned}
     1.& \quad \text{for each $t\in I$} \quad \eta(t,\cdot) \in L^1_{\rho_0}(X)\cap L^2_{\rho_0}(X),\\
     2.& \quad \text{for each $t\in I$} \quad \dot{\eta}(t,\cdot) \in L^1_{\rho_0}(X)\cap L^2_{\rho_0}(X). 
    \end{aligned}
\end{equation}

\subsection{Convexity of kernels for the formulation (A)} \label{sec:convexity}
We provide additional explanation regarding the form of our interaction energy, specifically the non-local energies $\E_r$ and $\E_a$. We only consider radially symmetric kernels. When these interaction energies are employed in the context of  gradient flows, a radially decreasing kernel ${K}_r$ is referred to as the \textit{repulsive} kernel, whereas a radially increasing kernel ${K}_a$  is referred to as the \textit{attractive} kernel. We keep this terminology.

When it comes to convexity, for the power law and logarithmic kernels, 
those that are convex in {\it radial variable} include
\begin{align*}
 &|y|^{p} \quad (p<0), \quad -\log(|y|), \quad -|y|^p \quad (0\le p \le 1)  & &\text{for repulsive cases and}\\
 &|y|^q \quad (q\ge 1)  &&\text{for attractive cases}.
\end{align*}
Here is a subtle point regarding the convexity of the kernel: For the repulsive kernels mentioned above, the kernel is not convex in the variable $y$, but rather in the radial variable $|y|$. Conversely, for the attractive kernels, the mapping $y \mapsto |y|^q$ is convex in $y$ for $q\ge 1$.

To make use of the relative entropy to obtain a stability estimate, we require the interaction energy to be strictly convex in its state variables, and the convexity  of the interaction energy is directly linked to the convexity of the kernel in $y$. As a result, the basic  relative entropy method is not applicable for the above-mentioned repulsive kernels. Notably, these cases include the Newtonian kernel $|y|^{2-d}$ for $d\ge 3$. The same requirement of being convex in $y$, not in $|y|$, is assumed in the seminal paper \cite{1997McCann}, where the concept of displacement convexity was introduced. The advanced use of the relative entropy method in multi-dimensions, where the lack of convexity leads to non-definite relative energies that are nevertheless managed by the large friction term, can be found in \cite{2021Choia}.

With that said, in this section, we present how the lack of convexity in repulsive cases can be resolved in one dimension. This is a consequence of $\eta$ being a monotone function.

The line $\left\{(x,x') ~|~ x=x'\right\}$ is $\rho_0\times\rho_0$-negligible under the assumption that $\rho_0$ is absolutely continuous with respect to the Lebesgue measure. Thus, we can express for instance $\E_r[\eta]$ of \eqref{ErEa} as
\begin{align*}
 &\iint_{x\ge x'} \Big({K}_r\big(\eta(x)-\eta(x')\big) -  {{K}_r\big(\eta_*(x)-\eta_*(x')\big)}\Big) d\rho_{0}(x)d\rho_0(x')\\
 + \: & \iint_{x\le x'} \Big({K}_r\big(\eta(x)-\eta(x')\big) -  {{K}_r\big(\eta_*(x)-\eta_*(x')\big)}\Big) d\rho_{0}(x)d\rho_0(x').
\end{align*}
Note that in the first integral, $\eta(x)-\eta(x')\ge 0$ due to the monotonicity of $\eta\in {\textsf{S}}(X)$. On the other hand, in the second integral, $\eta(x) - \eta(x')\le 0$. However, the calculation of the second integral is achieved via a change of variables, namely $z=x'$ and $z'=x$
\begin{align*}
&\iint_{z\ge z'} \Big({K}_r\big(\eta(z)-\eta(z')\big) -  {{K}_r\big(\eta_*(z)-\eta_*(z')\big)}\Big) d\rho_{0}(z)d\rho_0(z'),
\end{align*}
where radial symmetry of the kernel is employed in this derivation. Therefore, the energy can be calculated solely with the contribution from
$$2\iint_{x\ge x'} \Big({K}_r\big(\eta(x)-\eta(x')\big) -  {{K}_r\big(\eta_*(x)-\eta_*(x')\big)}\Big) d\rho_{0}(x)d\rho_0(x')$$
where the kernels are only evaluated at those $x, x'$ where $\eta(x)-\eta(x')\ge 0$. Thus, only the convexity of $K$ in the radial variable matters. This calculation does not generalize to higher dimensions. 

In one space dimension, if $p\le -1$, then $|y|^{p}$ is not locally integrable and we do not perceive this case to make much sense and is omitted from our consideration. A logarithmic kernel will also not be considered in this paper. 

Finally,  subtracting $K\big(\eta_*(x) - \eta_*(x'))$ ensures that the integral remains finite even if the domain $\tilde{X}$ is not bounded. We first remark that this subtraction does not affect the calculations of the first variation of energy
$$ \frac{d}{ds} \E[\eta +s\phi] \Big|_{s=0}$$
for any suitable test function $\phi$. Note that $|\eta(x)-\eta(x')|^q$ becomes large if $|x-x'|$ is large. In a bounded region of $X \times X$ with our choices in \eqref{kernellaw}, both
$|\eta(x) - \eta(x')|^p$ and $|\eta(x) - \eta(x')|^q$ are always integrable, so $\eta$ being of finite energy has nothing to do with $\eta_*$ in a bounded region. Therefore, fixing $\eta_*$ and requiring $\eta$ to have finite energy  assigns asymptotic boundary conditions as $|x| \rightarrow \infty$. 

\section{Relative Hamiltonian computation for the Euler-Poisson system in one space dimension} \label{sec:EP}
In this section, we  justify \eqref{dotrel} for the Euler-Poisson system in one space dimension. We use the formulation (A) in Section \ref{sec:(A)}.

\subsection{Dissipative weak solution and strong solution}
We define weak and strong flows for equation \eqref{Ham}
for the Hamiltonian
\begin{equation}\label{H}
\begin{aligned} 
 &\H[\eta,m] = \K[m] + \E[\eta], \quad \K[m] = \int_X \frac{m^2}{2\rho_0} \: dx, \\
 &\text{$\E$ is given by \eqref{E}-\eqref{Ep}.}
\end{aligned}
\end{equation}

Suppose $\eta\in \hat{\textsf{S}}(X)$ is admissible. Since $\eta$ is assumed to be bi-Lipschitz, as is the fixed map $\eta_*$, for $p$ and $q$ in \eqref{kernellaw} and $\gamma$ in \eqref{Ep}
there exist constants $E_1$, $E_2$, and $E_3$ such that 
\begin{equation} \label{suff}
    \begin{aligned}
    &\iint_{X \times X} |x-x'|^{p-1}\Big| \big(\eta(x)-\eta(x')\big) -\big(\eta_*(x)-\eta_*(x')\big)\Big| d\rho_0(x)d\rho_0(x') < E_1,\\
    &\iint_{X \times X} |x-x'|^{q-1}\Big| \big(\eta(x)-\eta(x')\big) -\big(\eta_*(x)-\eta_*(x')\big)\Big| d\rho_0(x)d\rho_0(x') < E_2,\\
    &\int_X \left|\frac{\eta_x(x) - \eta_{*x}(x)}{\rho_0}\right| \: \dr < E_3.
    \end{aligned}
\end{equation}
To be specific, $E_1$ and $E_3$ depend on  $\Lip(\eta)$ and parameters and  $E_2$ depends on $\Lip(\eta^{-1})$ and parameters. The detailed calculations for these are provided in the Appendix, see Lemma \ref{lem:const}

Given these observations, for the weak formulation of a Hamiltonian flow, we consider functions $\phi\in C^{0,1}(X)$ satisfying 
\begin{equation} \nonumber
 J_\alpha[\phi]:=\iint_{X\times X} |x-x'|^{\alpha-1}~| \phi(x) - \phi(x')| \:d\rho_0(x)d\rho_0(x') < \infty \quad \text{for $\alpha=p,q$} \label{phicond}
\end{equation}
and such that $\frac{\phi_x}{\rho_0}$ is $\rho_0$-integrable, as $\rho_0$ a.e. defined function.
For such a $\phi$, we define {\it virtual work} as the directional derivative
$$ \phi \mapsto -\frac{d}{ds} \E[\eta + s\phi] \Big|_{s=0}.$$
\begin{proposition} \label{continuity}
 Suppose $\E$ is given by the formula \eqref{H} and $\eta\in \hat{\textsf{S}}(X)$ is admissible. Then the functional $\mathbf{f}[\eta] : \phi \mapsto -\frac{d}{ds} \E[\eta + s\phi] \Big|_{s=0}$ is continuous on 
 $$\displaystyle\mathcal{C}=\left\{\phi \in C^{0,1}(X) ~|~ J_p[\phi] < \infty, \quad J_q[\phi] < \infty, \quad \text{$\frac{\phi_x}{\rho_0}$ is $\rho_0$-integrable}\right\}$$ and has the representation
 \begin{equation*} 
 \begin{aligned}
  &\mathbf{f}[\eta](\phi)=-\int_X (1-\gamma) \left(\frac{\eta_x}{\rho_0}\right)^{-\gamma}\left(\frac{\phi_x}{\rho_0}\right) \:d\rho_0(x)\\
  &+\iint_{X\times X} \sgn\Big(\eta(x) - \eta(x')\Big)\Big|\eta(x) - \eta(x')\Big|^{p-1}\Big(\phi(x)-\phi(x')\Big)\: d\rho_0(x)d\rho_0(x')\\
  &-\iint_{X\times X} \sgn\Big(\eta(x) - \eta(x')\Big)\Big|\eta(x) - \eta(x')\Big|^{q-1}\Big(\phi(x)-\phi(x')\Big)\: d\rho_0(x)d\rho_0(x').
 \end{aligned}
 \end{equation*}
\end{proposition}
\begin{proof}
Since $\eta^{-1}$ and $\phi$ are Lipschitz, and $\Big|\frac{\phi_x}{\eta_x}\Big| \le \Lip(\eta^{-1})\Lip(\phi)$, we can take $\eps_0>0$ so small that $\Big|s\frac{\phi_x}{\eta_x}\Big| \le \frac{1}{2}$ for $s\in(-\eps_0,\eps_0)$.
To use the Lebesgue dominated convergence theorem, we note that for $s\in(-\eps_0,\eps_0)$
\begin{align*}
    &\int_X \left|\frac{1}{s} \left\{ \Big(\frac{\eta_x + s\phi_x}{\rho_0}\Big)^{1-\gamma} - \Big(\frac{\eta_x}{\rho_0}\Big)^{1-\gamma} \right\} \right|\: \dr\\
 &= \int_X \left|\Big(\frac{\eta_x}{\rho_0}\Big)^{1-\gamma} \frac{1}{s}\int_0^s(1-\gamma)\Big(1 + \lambda\frac{\phi_x}{\eta_x}\Big)^{-\gamma}\left(\frac{\phi_x}{\eta_x}\right)\:d\lambda \right|\: \dr\\
 &\le 2^{\gamma}|1-\gamma|\int_X \Big(\frac{\eta_x}{\rho_0}\Big)^{1-\gamma} \left|\frac{\phi_x}{\eta_x}\right| \: \dr= 2^{\gamma}|1-\gamma|\int_X \Big(\frac{\eta_x}{\rho_0}\Big)^{-\gamma} \left|\frac{\phi_x}{\rho_0}\right| \: \dr.
\end{align*}
In addition, $\Big(\frac{\eta_x}{\rho_0}\Big)^{-\gamma}$ is bounded $\rho_0$-a.e. because $\eta^{-1}$ is Lipschitz and $\frac{\phi_x}{\rho_0} \in L^1_{\rho_0}(X)$. Thus the last integral is finite. Passing to the limit,
 \begin{align*}
  -\lim_{s \rightarrow 0+} \frac{1}{s} \int_X \Big(\frac{\eta_x + s\phi_x}{\rho_0}\Big)^{1-\gamma} -\Big(\frac{\eta_x}{\rho_0}\Big)^{1-\gamma} \: \dr  = -(1-\gamma)\int_X \left(\frac{\eta_x}{\rho_0}\right)^{-\gamma}\left(\frac{\phi_x}{\rho_0}\right) \: \dr.
 \end{align*}
For the non-local components, we have $\Big|\frac{\phi(x)-\phi(x')}{\eta(x) -\eta(x')}\Big| \le \Lip(\eta^{-1})\Lip(\phi)$, and we can take $\eps_0>0$ so small that $\Big|s\frac{\phi(x)-\phi(x')}{\eta(x) -\eta(x')}\Big| \le \Lip(\eta^{-1})\Lip(\phi) \le \frac{1}{2}$ for $s\in(-\eps_0,\eps_0)$. For $s \in (-\eps_0,\eps_0)$
\begin{align*}
  &\iint_{X\times X} \left|\frac{1}{s}K_r\Big( (\eta + s\phi)(x) - (\eta + s\phi)(x')\Big)- \frac{1}{s}K_r\Big(\eta(x) - \eta(x')\Big)\right| \: d\rho_0(x)d\rho_0(x') \\
  \le &\iint_{X\times X}\left|\frac{1}{s}\int_0^s \Big| (\eta + \lambda\phi)(x) - (\eta + \lambda\phi)(x')\Big|^{p-1}\: d\lambda\Big(\phi(x)-\phi(x')\Big)\right|\: d\rho_0(x)d\rho_0(x')\\ 
  = &\iint_{X\times X}\left|\frac{1}{s}\int_0^s \Big| \frac{1}{\eta(x)-\eta(x')}\Big|^{1-p} \Big| \frac{1}{1 + \lambda\frac{\phi(x)-\phi(x')}{\eta(x)-\eta(x')}}\Big|^{1-p}\: d\lambda\Big(\phi(x)-\phi(x')\Big)\right|\: d\rho_0(x)d\rho_0(x')\\ 
  \le &(2\Lip(\eta^{-1}))^{1-p}\iint_{X\times X} |x-x'|^{p-1} |\phi(x)-\phi(x')| \: d\rho_0(x)d\rho_0(x') < \infty
\end{align*}
because of \eqref{suff}. Taking the limit $s \rightarrow 0$ gives the result. Analogous arguments apply to $K_a$. Continuity follows easily from the representation. 
\end{proof}

Having established the first variation of the interaction energy, we proceed to define the weak and strong Hamiltonian flow.
\begin{definition} 
    Let $I=[a,b)$, $\eta \in \hat{\mathcal{A}}(I \times X)$, ${v}=\dot\eta$, ${m}= \rho_0 {v}$, and let the Hamiltonian $\H$ be given by \eqref{H}. We say that $\eta$ is a \textit{dissipative weak Hamiltonian flow} for $\H$ on $I$ if the following holds: 
    \begin{equation} \label{wksol}
        \begin{aligned}
         -&\int_I\int_{X}^{}{m}\dot{\phi}(t,x) \: dxdt = \int_I \mathbf{f}{[\eta(t,\cdot)]}\big(\phi(t,\cdot)\big) \: dt
        \end{aligned}
    \end{equation}
    for all $\phi\in C_c^{0,1}((a,b) \times X)$, and it also satisfies
    \begin{equation} \label{dissip}
        \begin{aligned}
        &-\int_I \H[\eta, {m}]\dot{\theta}(t)dt \leq 0 \quad &&\text{for any non-negative $\theta\in C^{0,1}_c((a,b))$ and} \\
        &\H[\eta,m](a+) \le \H[\eta,m](a).
        \end{aligned}
    \end{equation}
\end{definition}
\begin{remark} The condition $\eqref{dissip}$ implies that $t \mapsto \mathcal{H}[\eta,m](t)$ is decreasing on the interval $[a,b)$.
\end{remark}

\begin{remark} \label{wktest}
 Proposition \ref{continuity} shows that the equality in \eqref{wksol} holds for any $\phi=\varphi\theta$ with $\varphi \in \mathcal{C}$ and $\theta\in C_0^{0,1}((a,b))$. In particular, for $\eta$, $\tilde\eta$ in $\hat{\mathcal{A}}(I\times X)$, \eqref{wksol} holds for $\phi= (\eta- \tilde{\eta})\theta$.
\end{remark}
\begin{definition}
    We define a dissipative weak Hamiltonian flow $\bar\eta \in \hat{\mathcal{A}}(I\times X)$ to be a \textit{strong Hamiltonian flow} for $\H$ on $I$ if the following conditions are satisfied:
    \begin{enumerate}[label=(\roman*)]
    \item For each $t\in I$, 
    $$\ddot{\bar\eta}(t,\cdot), \dot{\bar{\eta}}_x(t,\cdot), \bar\eta_{xx}(t,\cdot) \in L^\infty_{\rho_0}, \quad J_p[{\dot{\bar\eta}}](t),  J_q[{\dot{\bar\eta}}](t) <\infty, \quad \frac{\dot{\bar{\eta}}_x(t,\cdot)}{\rho_0} \in L^1_{\rho_0}(X).$$ 
    \item For all $\phi\in C_c^{0,1}(X)$, and for a.e. $t\in I$,
    \begin{equation} \nonumber
    \int_X \dot{\bar{{m}}} \phi(x) \: dx = \mathbf{f}[\bar\eta](\phi).
    \end{equation}
    \item The two inequalities in $\eqref{dissip}$ are satisfied as equalities.
    \end{enumerate}
\end{definition}
Throughout the Section \ref{sec:EP}, we fix $\H$ as in \eqref{H} and $I=[a,b)$ and assume
\begin{equation} \label{hyp}
\begin{aligned}
    &\text{$\eta$ is a dissipative weak  flow for $\H$ on $I$,}\\
    &\text{$\bar\eta$ is a strong flow for $\H$ on $I$,}\\
    &\text{${v}=\dot\eta$, $\bar{v}=\dot{\bar{\eta}}$, ${m}= \rho_0 {v}$, and $\bar{m}= \rho_0 \bar{v}$.} 
\end{aligned}
\end{equation}

\subsection{Convexity and lower bound of relative Hamiltonian} \label{sec:rel}

We define relative energies between functions $\eta(t,\cdot)$, $\bar\eta(t,\cdot)$, $m(t,\cdot)$, and $\bar m (t,\cdot)$ for each time $t \in I$.
\begin{align} 
    \E_{w}[\eta | \bar{\eta}] 
     & \  :=\int_{X}^{}\psi(\tau) - \psi(\bar\tau) - \psi'(\bar\tau)\left( \tau - \bar{\tau} \right)d\rho_0(x) \label{exp1}\\
     & \ = \int_X\psi(\tau |\bar{\tau} ) d\rho_0(x), \quad \text{where $\psi(\tau)=\tau^{1-\gamma}$,}\nonumber \\
    \mathcal{E}_r[\eta | \bar{\eta}] 
    & \  :=\iint_{X \times X}{K}_r\big(\eta(x)-\eta(x')\big) - {K}_r\big(\bar{\eta}(x)-\bar{\eta}(x')\big) \nonumber\\ 
    & \ \quad - {K}_r'\big(\bar{\eta}(x) - \bar{\eta}(x')\big)\big( \eta(x) - \eta(x') - \bar{\eta}(x) + \bar{\eta}(x') \big) \: d\rho_0(x)d\rho_0(x') \label{exp2}\\
    & \ = \iint_{X \times X} \mathcal{K}_r(\eta | \bar{\eta}) \: d\rho_0(x)d\rho_0(x'), \nonumber\\
    \mathcal{E}_a[\eta | \bar{\eta}] 
    & \ :=\iint_{X \times X}{K}_a\big(\eta(x)-\eta(x')\big) - {K}_a\big(\bar{\eta}(x)-\bar{\eta}(x')\big) \nonumber \label{exp3}\\ 
    & \ \quad - {K}_a'\big(\bar{\eta}(x) - \bar{\eta}(x')\big)\big( \eta(x) - \eta(x') - \bar{\eta}(x) + \bar{\eta}(x') \big) \: d\rho_0(x)d\rho_0(x') \\
    & \ = \iint_{X \times X}{K}_a\big(\eta | \bar{\eta}\big) \: d\rho_0(x)d\rho_0(x') \nonumber\\
    \K[{m} | \bar{m}] 
     &\ :=\int_X \frac{|{m}|^2}{2\rho_0}- \frac{|\bar{m}|^2}{2\rho_0} - \frac{\bar{m}}{\rho_0}({m}-\bar{m})\:dx \nonumber\\ 
     &\ = \int_X \frac{|{m} - {\bar{m}|}^2}{2\rho_0}dx  = \int_X \frac{\rho_0 |{v} - {\bar{v}}|^2}{2} \: dx, \label{expK}
\end{align}
The relative Hamiltonian is given as
\begin{equation} \label{relHam}
    \begin{aligned}
    \H[\eta, {m}| \bar\eta, \bar{m}]&:= 
    \K[{m} | \bar{m}] + \E_{w}[\eta | \bar{\eta}] + \mathcal{E}_r[\eta | \bar{\eta}] + \mathcal{E}_a[\eta | \bar{\eta}].
    \end{aligned}
\end{equation}
Along with the definitions, we have introduced notations for the abbreviations of the integrands $\psi(\tau|\bar\tau)$, $K_r(\eta|\bar\eta)$, and $K_a(\eta|\bar\eta)$. We also define the notation
$$[f](x,x'):=f(x)-f(x')$$ 
for a function $f$ on $X$ for simplicity.

Now, we proceed to estimate the relative energies by calculating their upper and lower bounds. Since the dissipative weak solution $\eta \in \hat{\mathcal{A}}(I\times X)$ is bi-Lipschitz, even when it comes to the weak flow $\eta$, $||D\eta||_\infty$ and $||D\eta^{-1}||_\infty$ are available. Although avoding the use of these makes computations more cumbersome, we present estimates without using these whenever feasible.

\begin{proposition}[Finiteness of relative energies] \label{ubnd}
    Assume \eqref{hyp}. Then the expressions \eqref{exp1}, \eqref{exp2}, and \eqref{exp3} are all finite.
\end{proposition} 
\begin{proof} 
    To validate finiteness of \eqref{exp2} and \eqref{exp3}, it suffices to observe that
\begin{align*}
 &\left|K_r'\big([\bar\eta](x,x')\big)[\eta-\bar\eta](x,x')\right| \\
 &= \big|[\bar\eta]\big|^{p-1} \big|[\eta-\bar\eta]\big|
 \\
 &\le ~ \Lip(\bar\eta^{-1}(t,\cdot))^{1-p} |x-x'|^{p-1}\big|[\eta-\eta_*]\big| + \big|[\bar\eta-\eta_*]\big|,
\end{align*}
and    
\begin{align*}
&\left|K'_a([\bar\eta])[\eta-\bar\eta]\right| \le \Lip(\bar\eta(t,\cdot))^{q-1}\big|[\eta-\eta_*]\big| + \big|[\bar\eta-\eta_*]\big|.
\end{align*}
The double integrals of the two above expressions are finite by \eqref{suff}. To validate finiteness of \eqref{exp1}, it suffices to observe that
\begin{align*}
 \left|\left(\frac{\rho_0}{\bar\eta_x}\right)^{\gamma} \left(\frac{\eta_x - \bar\eta_x}{\rho_0}\right)\right| \le C|\Lip(\bar\eta^{-1}(t,\cdot))|^\gamma \left|\frac{\eta_x - \bar\eta_x}{\rho_0}\right|,
\end{align*}
which is $\rho_0$-integrable by \eqref{suff}.
\end{proof}

Next, we make use of the convexity of the interaction energy to conduct a  calculation similar to \cite[Lemma 1]{2016Giesselmann}.

\begin{theorem}\label{lowerbound} Assume \eqref{hyp}. Then there exist positive constants $c_1$, $c_2$, and $c_3$, which depend only on $p$, $q$, $\gamma$, $\Lip(\bar\eta)$, and $\Lip(\bar\eta^{-1})$, satisfying the following for every $t\in I$: 
\begin{enumerate}[label={(\arabic*)}] \setcounter{enumi}{-1} 
 \item If $q=1$, then $\E_a[\eta | \bar\eta] = 0$. If $p=1$, then $\E_r[\eta | \bar\eta] = 0$. 
 \item If $q \in (1,2)$,
       \begin{equation} \label{Lq1}
       \begin{aligned}
           &\E_a[\eta|\bar\eta] \ge c_1 \left(\iint_{\frac{\eta(x)-\eta(x')}{x-x'} \le 3\Lip(\bar\eta)}\big|[\eta - \bar\eta]\big|^2 |x-x'|^{q-2}\:d\rho_0(x)d\rho_0(x')\right. \\
           & ~~\quad \quad \quad\quad+\left.\iint_{\frac{\eta(x)-\eta(x')}{x-x'} > 3\Lip(\bar\eta)}\big|[\eta - \bar\eta]\big|^q d\rho_0(x)d\rho_0(x')\right).
       \end{aligned}
       \end{equation}
 \item If $q\ge 2$,
       \begin{equation*}
        \E_a[\eta|\bar\eta] \ge c_2 \iint_{X \times X} \big|[\eta - \bar\eta]\big|^q d\rho_0(x)d\rho_0(x'). 
       \end{equation*}
 \item If $p\in (0,1)$,
 \begin{equation*} 
  \begin{aligned}
   &\E_r[\eta|\bar\eta] \ge c_3 \left(\iint_{\frac{\eta(x)-\eta(x')}{x-x'} \le 3\Lip(\bar\eta)}\big|[\eta - \bar\eta]\big|^2 |x-x'|^{p-2}\:d\rho_0(x)d\rho_0(x')\right. \\
 & ~~\quad \quad \quad\quad+\left.\iint_{\frac{\eta(x)-\eta(x')}{x-x'} > 3\Lip(\bar\eta)}\big|[\eta - \bar\eta]\big|^p d\rho_0(x)d\rho_0(x')\right).
  \end{aligned}
 \end{equation*}
\end{enumerate}
Furthermore, there exist positive constants $c_4$, and $c_5$ such that the following estimates hold for every $t\in I$. The constants depend on $p$, $q$, $\gamma$, $\Lip(\bar\eta)$, $\Lip(\bar\eta^{-1})$, and additionally on $\Lip(\eta)$, $\Lip(\eta^{-1})$, and the parameters $\beta\in (0,1)$ and  $\beta'\in(-1,0)$ below.
\begin{enumerate}[label={(\arabic*)}]
    \setcounter{enumi}{3} 
 \item For any $\beta \in(0,1)$ and $r(\beta):=\frac{\gamma}{1-\beta}>1$,
 \begin{equation*}
 \E_{w}[\eta|\bar\eta] \ge c_4 \left(\int_{\tau \le 3\Lip(\bar\eta)} (\tau - \bar\tau)^2 \:\dr + \left(\int_{\tau > 3\Lip(\bar\eta)} |\tau - \bar\tau|^\beta \: \dr\right)^{r(\beta)}\right). 
 \end{equation*}
\item If $p\in(-1,0)$, for any $\beta' \in(-p,1)$ and $r'(\beta'):=\frac{1-p}{1-p-\beta'}>1$
 \begin{equation*}
  \begin{aligned}
   &\E_r[\eta|\bar\eta] \ge c_5 \Bigg(\iint_{\frac{\eta(x)-\eta(x')}{x-x'} \le 3\Lip(\bar\eta)}\big|[\eta - \bar\eta]\big|^2 |x-x'|^{p-2}\:d\rho_0(x)d\rho_0(x') \\
 & ~~ \quad+\left(\iint_{\frac{\eta(x)-\eta(x')}{x-x'} > 3\Lip(\bar\eta)}\big|[\eta - \bar\eta]\big|^{\beta'+p}|x-x'|^{-\beta'} d\rho_0(x)d\rho_0(x')\right)^{r'(\beta')}\Bigg).
  \end{aligned}
 \end{equation*}
\end{enumerate}
\end{theorem}
\begin{proof}
(Assertion 0). Let $q=1$ and $p=1$. Then 
\begin{align*}
 \E_a[\eta|\bar\eta] = -\E_r[\eta|\bar\eta]&= \iint |[\eta]| - |[\bar\eta]| -\sgn([\bar\eta])[\eta-\bar\eta] \:\drr\\
  &= \iint [\eta]\Big(\sgn([\eta]) - \sgn([\bar\eta])\Big)\:\drr = 0.
\end{align*}
In the last equality, we used that $[\eta]\ne 0 \Longrightarrow \sgn([\eta])=\sgn([\bar\eta]) = \sgn(x-x')$.

(Assertion 1). We introduce for each $(x,x')$, $x\ne x'$ abbreviations 
$$ h(x,x') = \frac{\eta(x) - \eta(x')}{x-x'}, \quad \bar{h}(x,x') = \frac{\bar\eta(x) - \bar\eta(x')}{x-x'}.$$
By monotonicity of $\eta$ and $\eta_*$, $h(x,x')$ and $h_*(x,x')$ are nonnegative, and
 \begin{align*}
  \big|[\eta](x,x')\big| = h(x,x')|x-x'|, \quad \big|[\eta_*](x,x')\big| = h_*(x,x')|x-x'|.
 \end{align*}
For $\rho_0\times\rho_0$-a.e. $(x,x')$,
\begin{align*}
K_a(\eta|\bar\eta) &=\frac{1}{q}\left(h^q - \bar h^q - q\bar h^{q-1}(h-\bar h)\right)|x-x'|^q \\
&= (q-1) \int_0^1 \int_0^s (h-\bar h)^2(h_\lambda)^{q-2} \: d\lambda ds\, |x-x'|^q,
\end{align*}
where $h_\lambda :=\lambda h + (1-\lambda) \bar h$  is nonnegative too. We divide the domain $X\times X$ into two regions, one where $h \le 3\Lip(\bar\eta)$ and the other where $h > 3\Lip(\bar\eta)$. If $h \le 3\Lip(\bar\eta)$ then $h_\lambda \le 3\Lip(\bar\eta)$ and consequently $$h_\lambda^{q-2}\ge (3\Lip(\bar\eta))^{q-2}.$$ If $h> 3\Lip(\bar\eta)$ then $h \ge h_\lambda$ and thus using $q-2<0$
\begin{align*}
 h_\lambda^{q-2}(h-\bar h)^2 
\ge h^{q-2}(h-\bar h)^2 = \left(\frac{h-\bar h}{h}\right)^{2-q} (h-\bar h)^{q} \ge \left( \frac{1}{2} \right)^{2-q} (h-\bar h)^{q}.
\end{align*}
In conclusion, we obtain \eqref{Lq1}.

(Assertion 2). In case $q \ge 2$, we have
\begin{align*}
 (h-\bar h)^2 (h_\lambda)^{q-2} &= |h-\bar h|^q \left(\frac{h_\lambda}{|h-\bar h|}\right)^{q-2} \ge |h-\bar h|^q\left\{
 \begin{aligned}
  (1-\lambda)^{q-2} \quad \text{if $\bar h \ge h$,}\\
  \lambda^{q-2} \quad \text{if $\bar h < h$}.
 \end{aligned}\right.
\end{align*}
By taking $\displaystyle c_3:= (q-1)\min\left\{ \int_0^1\int_0^s \lambda^{q-2} \:d\lambda ds, \int_0^1\int_0^s (1-\lambda)^{q-2} \:d\lambda ds\right\} >0$, the assertion follows.

(Assertion 3). Analogously to the proof of Assertion 1 we obtain
\begin{align}
K_r(\eta|\bar\eta) &=-\frac{1}{p}\left(h^p - \bar h^p - p\bar h^{p-1}(h-\bar h)\right)|x-x'|^p \nonumber\\
&= (1-p) \int_0^1 \int_0^s (h-\bar h)^2(h_\lambda)^{p-2} \: d\lambda ds\, |x-x'|^p. \label{pest}
\end{align}

(Assertion 4). We estimate $\int_{\tilde X} \psi(\tau|\bar\tau) \:dx$, where $\tilde{X}$ is an open set such that $\rho_0 = \mathcal{L}\mres \tilde{X}$. We write
\begin{align*}
 \psi(\tau|\bar\tau) = (1-\gamma)(-\gamma)\int_0^1\int_0^s (\lambda \tau + (1-\lambda) \bar\tau)^{-1-\gamma} (\tau - \bar\tau)^2 \:d\lambda ds.
\end{align*}
Note that $\tau_\lambda:=\lambda \tau + (1-\lambda)\bar\tau$ is nonnegative $\rho_0$-a.e.. If $\tau \le 3\Lip(\bar\eta)$, then $\tau_\lambda \le 3\Lip(\bar\eta)$ and consequently $$(\tau_\lambda)^{-1-\gamma} \ge (3\Lip(\bar\eta))^{-1-\gamma}.$$
In case $\tau > 3\Lip(\bar\eta)$, then $\tau\ge \tau_\lambda$ and thus
\begin{align*}
 \tau_\lambda^{-1-\gamma}(\tau-\bar\tau)^2 \ge \tau^{-1-\gamma}(\tau-\bar\tau)^2
\ge \left(\frac{\tau-\bar\tau}{\tau}\right)^{1+\gamma} (\tau-\bar\tau)^{1-\gamma} \ge \left( \frac{1}{2} \right)^{1+\gamma} (\tau-\bar\tau)^{1-\gamma}.
\end{align*}
We apply the reverse H\"older inequality \cite[Theorem 2.12]{2003Adams} to the integral in the last expression: For any $0<\beta <1$, let  
$$s = \frac{1-\beta}{\gamma} \in (0,1), \quad \delta = \frac{\beta}{s}>0 \quad \text{so that $\frac{(1-\gamma-\delta)s}{s-1} = 1$}.$$
We check that
\begin{align*}
    &\int_{\tilde{X} \cap \{\tau>3\Lip(\bar\eta)\}} (\tau-\bar\tau)^{1-\gamma} \:dx < \infty \quad \text{by Propostion \ref{ubnd}};\\
    &\text{if  $g:= (\tau - \bar\tau)^{1-\gamma-\delta}$ then $g\ne 0$ in $\tilde{X} \cap \{\tau>3\Lip(\bar\eta)\}$};\\
    &\int_{\tilde{X} \cap \{\tau>3\Lip(\bar\eta)\}} (\tau-\bar\tau)^{(1-\gamma-\delta) \frac{s}{s-1}} \:dx = \int_{\tilde{X} \cap \{\tau>3\Lip(\bar\eta)\}} (\tau-\bar\tau) \:dx \le E_3 < \infty 
\end{align*}
by \eqref{suff}.
The reverse H\"older inequality gives that
\begin{align*}
 \int_{\tilde{X} \cap \{\tau>3\Lip(\bar\eta)\}} (\tau-\bar\tau)^{1-\gamma} \:dx 
 &\ge \frac{\left(\int_{\tilde{X} \cap \{\tau>3\Lip(\bar\eta)\}} (\tau-\bar\tau)^{\delta s} \:dx \right)^{\frac{1}{s}}}{\left(\int_{\tilde{X} \cap \{\tau>3\Lip(\bar\eta)\}} (\tau-\bar\tau)^{(1-\gamma-\delta) \frac{s}{s-1}} \:dx \right)^{\frac{1-s}{s}}}\\
 & \ge \frac{\left(\int_{\tilde{X} \cap \{\tau>3\Lip(\bar\eta)\}} (\tau-\bar\tau)^{\delta s} \:dx \right)^{\frac{1}{s}}}{\left(\int_{\tilde{X}} |\tau-\bar\tau|^{(1-\gamma-\delta) \frac{s}{s-1}} \:dx \right)^{\frac{1-s}{s}}}
\end{align*}
and the assertion follows.

(Assertion 5). We consider \eqref{pest} for $p \in (-1,0)$. In the region where $h \le 3\Lip(\bar\eta)$ the integral is treated similarly as in assertion 3. In case $h > 3\Lip(\bar\eta)$ we use the reverse H\"older inequality. For any choice $-p<\beta' <1$, let
$$ s = \frac{1-p-\beta'}{1-p} \in (0,1), \quad \delta  = \frac{\beta' + p}{s} >0,$$
which implies that
$$ \frac{(p-\delta)s}{s-1} = 1, \quad (p-\delta)(p-1)s = \beta'.$$
Then, we have
\begin{align*}
&\iint_{h > 3\Lip(\bar\eta)} \big| [\eta - \bar\eta]\big|^p \: \drr\\
\ge &\frac{\left(\iint_{{\tilde{X}}^2\cap \{ h >3\Lip(\bar\eta)\}} \left(\big|[\eta - \bar\eta]\big|^\delta |x-x'|^{-(p-\delta)(p-1)}\right)^s \:dx \right)^{\frac{1}{s}}}{\left(\iint_{{\tilde{X}}^2 \cap \{h>3\Lip(\bar\eta)\}} \left(\big|[\eta - \bar\eta]\big|^{p-\delta}|x-x'|^{(p-\delta)(p-1)}\right)^{ \frac{s}{s-1}} \: \drr \right)^{\frac{1-s}{s}}}.
\end{align*}
The denominator is bounded by a constant dependent on $E_1$, which in turn depends on $\Lip(\eta^{-1})$.
\end{proof}

\begin{corollary} 
 \begin{enumerate}
  \item $\K[m|\bar m](t_0) = 0$ implies that $\dot{\eta}(t_0,\cdot) = \dot{\bar\eta}(t_0,\cdot)$ $\rho_0$-a.e.,
  \item $\E[\eta|\bar \eta](t_0)=0$ implies that $\eta(t_0,\cdot) = \bar\eta(t_0,\cdot) + C$ for any constant $C$ $\rho_0$-a.e..
 \end{enumerate}
\end{corollary}

\subsection{Relative Hamiltonian inequality} \label{sec:relham}
We now justify the calculations leading to equation \eqref{dotrel}. For $\eta$ and $\bar\eta$, which are a weak and a strong flow, respectively, we arrive at the {\it relative Hamiltonian inequality} as stated in Theorem \ref{ineqthm}. Before we prove Theorem \ref{ineqthm}, we establish the following estimates first.
\begin{proposition} \label{pp}
 Let $\eta,\bar\eta \in \hat{\textsf{S}}(X)$ be admissible and $\phi$ be Lipschitz. Then there exists a constant $C_1>0$ satisfying \eqref{e1} $\rho_0$-a.e., and constants $C_2,C_3>0$ satisfying \eqref{e2} and \eqref{e3}. The constants depend on $p$, $q$, $\gamma$, $\Lip(\phi)$, $\Lip(\bar\eta^{-1})$, and $\Lip(\eta^{-1})$. 
 \begin{align}
  \Big|\phi_x\Big(\psi'(\tau) - \psi'(\bar\tau) - \psi''(\bar\tau)(\tau - \bar\tau)\Big)\Big|  &\le C_1 \psi(\tau|\bar\tau), \label{e1}\\
  \Big|[\phi]\Big(K_r'([\eta]) - K_r'([\bar\eta]) - K_r''([\bar\eta])[\eta-\bar\eta]\Big)\Big| &\le C_2 K_r(\eta|\bar\eta), \label{e2}\\
   \Big|[\phi]\Big(K_a'([\eta]) - K_a'([\bar\eta]) - K_a''([\bar\eta])[\eta-\bar\eta]\Big)\Big| &\le C_3 K_a(\eta|\bar\eta) \label{e3}. 
 \end{align}
\end{proposition}
\begin{proof} For $\rho_0$-a.e. $x$,
 \begin{align*}
  &\phi_x\psi'(\tau|\bar\tau)=\phi_x \Big(\psi'(\tau) - \psi'(\bar\tau) - \psi''(\bar\tau)(\tau - \bar\tau)\Big)\\
  &= \phi_x(\tau - \bar\tau)^2\int_0^1 \int_0^s \psi'''(\lambda \tau + (1-\lambda)\bar\tau)\: d\lambda ds\\
  &= (\tau - \bar\tau)^2(-\gamma+1)(-\gamma)(-\gamma-1) \int_0^1\int_0^s (\tau^\lambda)^{-\gamma-1}~\frac{\phi_x}{\tau^\lambda} \: d\lambda ds,
 \end{align*}
 where $\tau^\lambda = \lambda \tau + (1-\eta) \bar\tau$. Observe that 
\begin{align*}
 \left| \frac{\phi_x}{\lambda \eta_x + (1-\eta) \bar\eta_x} \right| \le \max \left\{ \left| \frac{\phi_x}{\eta_x}\right|, \left| \frac{\phi_x}{\bar\eta_x}\right|\right\} \le \Lip(\phi) \big( \Lip(\eta^{-1}) + \Lip(\bar\eta^{-1})\big)
\end{align*}
and \eqref{e1} follows. Similarly,
\begin{align*}
 &[\phi]K_r'(\eta | \bar\eta)=
 [\phi]\Big(K_r'([\eta]) - K_r'([\bar\eta]) - K_r''([\bar\eta])[\eta-\bar\eta]\Big)\\
 &=-(p-1)(p-2)[\eta-\bar\eta]^2\int_0^1 \int_0^s \Big|[\eta^\lambda]\Big|^{p-2} \frac{[\phi]}{[\eta^\lambda]} \:d\lambda ds,
\end{align*}
and 
$$\left|\frac{[\phi]}{[\lambda \eta + (1-\lambda)\bar\eta]} \right| \le  \Lip(\phi) \big( \Lip(\eta^{-1}) + \Lip(\bar\eta^{-1})\big),$$
and \eqref{e2} follows. Equation \eqref{e3} follows similarly.
\end{proof}
\begin{theorem}[Relative Hamiltonian inequality] \label{ineqthm}
Assume \eqref{hyp}. Then the map $t \mapsto \H[\eta, m|\bar{\eta}, \bar{m}](t)$ is a function of Bounded Variation on $I$, and for any $t_1 \in (a,b)$, 
    \begin{equation} \label{relineq}
        \begin{aligned}
            \H[\eta, m|\bar{\eta}, \bar{m}] (t_1+) - \H[\eta, m| \bar{\eta}, \bar{m}] (a)&\le  \H[\eta, m|\bar{\eta}, \bar{m}] (t_1-) - \H[\eta, m| \bar{\eta}, \bar{m}] (a)\\ & \leq  \int_a^{t_1} \int_X \left(\frac{\bar{{v}}_x}{\rho_0}\right) \psi'(\tau | \bar{\tau}) \:\dr dt  \\
            &  + \int_{a}^{{t_1}} \iint_{X \times X} [\bar{{v}}](x, x'){K}_a'(\eta | \bar\eta)d\rho_0(x)d\rho_0(x') dt \\
            &  + \int_{a}^{{t_1}} \iint_{X \times X} [\bar{{v}}](x, x'){K}_r'(\eta | \bar\eta)d\rho_0(x)d\rho_0(x') dt.
        \end{aligned}
    \end{equation}
\end{theorem}

\begin{proof}[Proof of Theorem \ref{ineqthm}]
From \eqref{hyp} and \eqref{relHam}, we obtain that for any $\theta \in C_0^{0,1}((a,b))$ nonnegative
\begin{equation} \label{firsteqn}
 \begin{aligned}
        -\int_I &\H[\eta, {m} | \bar{\eta}, {\bar{m}}] \dot\theta(t)\:dt \\
        &\le -\int_I \bigg( - \left\langle \frac{\delta \K}{\delta {m}}(\bar{m}),{m} - \bar{m}\right\rangle -\left\langle \frac{\delta\E}{\delta \eta}(\bar{\eta}), \eta - \bar\eta\right\rangle\bigg)\dot\theta(t) \: dt\\
        & = \int_I\int_X \bar{v}({m}-\bar{{m}})\dot\theta dxdt\\
        & +\int_I\int_X \psi'(\bar\tau)(\tau - \bar\tau)\dot\theta \: \dr dt \\
        & + \int_I\iint_{X \times X} K'\big([\bar{\eta}](x, x')\big)\big([\eta](x, x') - [\bar{\eta}](x, x')\big)\dot\theta \: \drr dt.
\end{aligned}
\end{equation}
Asumming the regularities of strong solution $\bar\eta$, $\left|\frac{\bar{\eta}_x}{\rho_0}\right|^{-\gamma}\frac{|\eta_x - \bar\eta_x|}{\rho_0} \in L^1_{\rho_0}(X)$ and $(|x-x'|^{p-1} + |x-x'|^{q-1})|[\eta - \bar\eta]|\in L^1_{\rho_0\times\rho_0(X^2)}.$ Therefore we can write
\begin{align*}
        \int_I & \int_{X} \psi'(\bar\tau)(\tau - \bar\tau) \dot\theta \:\dr dt\\
        & = \int_I\int_X \Big( \big(\theta\psi'(\bar\tau)\big)^\bdot -  \theta\psi''(\bar\tau)\frac{\bar{v}_x}{\rho_0}\Big)(\tau - \bar\tau) \:\dr dt, \\
        \int_I & \iint_{X\times X} {K}'\big([\bar{\eta}]\big)[\eta-\bar{\eta}]\dot\theta \: \drr dt,\\
        & = \int_I \iint_{X\times X} \Big(\big(\theta K'([\bar\eta]) \big)^{\boldsymbol{\cdot}} - \theta K''([\bar\eta])[\bar v]\Big)[\eta - \bar\eta] \:\drr dt.\\
        \intertext{Also,}
        \int_I&\int_X \bar{v}({m}-\bar{{m}})\dot\theta dxdt = \int_I\int_X \Big(\big(\theta \bar{v}\big)^\bdot - \theta \dot{\bar{v}}\Big)({m}-\bar{{m}}) \: dxdt.
\end{align*}
Now, consider an approximation $(\eta^\epsilon - \bar\eta^\epsilon)_{\epsilon>0}$ such that for each $t\in I$, $(\eta^\epsilon-\bar\eta^{\epsilon})(t,\cdot)$, $(\dot{\eta}^\epsilon-\dot{\bar\eta}^{\epsilon})(t,\cdot)$, and $(\dot{\eta}_x^\epsilon-\dot{\bar\eta}_x^\epsilon)(t,\cdot)$ are in $C_c^\infty(X)$ and $(|x-x'|^{p-1} + |x-x'|^{q-1})[\eta^\epsilon - \bar\eta^\epsilon]$, $\frac{\eta_x^\epsilon - \bar\eta_x^\epsilon}{\rho_0}$, and $\dot{\eta}^\epsilon - \dot{\bar\eta}^\epsilon$ strongly converge to their repective limits in $L^1_{\rho_0\times\rho_0}$, $L^1_{\rho_0}$, and $L^2_{\rho_0}$. 

Now, $(\dot{\eta}^\epsilon - \dot{\bar\eta}^\epsilon)\theta$ can be used as test function for the strong solution, and thus
\begin{equation}\label{secondeqn}
\begin{aligned}
 \int_I\int_X -\theta \dot{\bar{v}}(m-\bar m) \:dx dt &= \lim_{\epsilon \rightarrow 0}  \int_I\int_X -\theta \dot{\bar{m}}(\dot\eta^\epsilon -\dot{\bar\eta}^\epsilon) \:dx dt\\
 &= \lim_{\epsilon \rightarrow 0}\int_I\int_X \psi'(\bar\tau)\theta\left(\frac{\dot{\eta}^\epsilon_x - \dot{\bar\eta}^\epsilon_x}{\rho_0}\right) \: \dr dt \\
 &+ \lim_{\epsilon \rightarrow 0}\int_I\iint_{X\times X} K'([\bar\eta])\theta  [\dot{\eta}^\epsilon-\dot{\bar\eta}^\epsilon] \:\drr dt\\
 &=\lim_{\epsilon \rightarrow 0}-\int_I\int_X \big(\theta\psi'(\bar\tau)\big)^\bdot \left(\frac{{\eta}^\epsilon_x - {\bar\eta}^\epsilon_x}{\rho_0}\right) \:\dr dt \\
 &+\lim_{\epsilon \rightarrow 0}-\int_I\iint_{X\times X} \big(\theta K'([\bar\eta])\big)^\bdot [\eta^\epsilon-\bar\eta^\epsilon] \:\drr dt \\
 &= -\int_I\int_X \big(\theta\psi'(\bar\tau)\big)^\bdot \left(\frac{{\eta}_x - {\bar\eta}_x}{\rho_0}\right) \:\dr dt \\
 &-\int_I\iint_{X\times X} \big(\theta K'([\bar\eta])\big)^\bdot [\eta-\bar\eta] \:\drr dt.
\end{aligned}
\end{equation}
After canceling out terms, the remaining terms are
\begin{align*}
 &\int_I\int_X \big(\theta \bar{v}\big)^\bdot ({m}-\bar{{m}}) \:dxdt-\int_I\theta \int_X \left(\frac{\bar{v}_x}{\rho_0}\right)\psi''(\bar\tau) (\tau - \bar\tau)\: \dr dt\\
 & -\int_I\iint_{X\times X}\theta  K''([\bar\eta])[\eta-\bar\eta] \:\drr dt.\\
 \intertext{By using $\theta\bar{v}$ as a test function for the dissipative weak solution we obtain}
 &= \int_I \theta \left[\int_X \left(\frac{\bar{v}_x}{\rho_0}\right) \Big( \psi'(\tau) - \psi'(\bar\tau) - \psi''(\bar\tau)(\tau - \bar\tau)\Big)\: \dr\right]  dt\\
 &+ \int_I \theta \left[\int_{X\times X} [\bar v] \Big( K'([\eta]) - K'([\bar\eta]) - K''([\bar\eta])[\eta-\bar\eta]\Big)\: \drr\right] dt.
\end{align*}
Considering the estimates \eqref{e1}-\eqref{e3} and Proposition \ref{ubnd}, it is evident that the two integrands in square brackets are uniformly bounded for $t \in I$. Additionally, it is important to note that the equalities from the third line of \eqref{firsteqn} hold for $\theta \in C_0^{0,1}(a,b)$, without the requirement of being nonnegative. Consequently, this implies that the mapping $t \mapsto \bigg( - \left\langle \frac{\delta \K}{\delta {m}}(\bar{m}),{m} - \bar{m}\right\rangle -\left\langle \frac{\delta\E}{\delta \eta}(\bar{\eta}), \eta - \bar\eta\right\rangle\bigg)(t)$ is absolutely continuous. Considering that $\bar\H(t)$ is a constant function and $\H(t)$ monotonically decreases, we conclude that $\allowbreak t \mapsto \H[\eta,m|\bar\eta,\bar m](t)$ is a function of bounded variation satisfying for any $t \in I$
\begin{equation} \label{jump}
    \H[\eta,m|\bar\eta,\bar m](t+) - \H[\eta,m|\bar\eta,\bar m](t-) = \H[\eta,m](t+) - \H[\eta,m](t-).
\end{equation}
Now, let $\epsilon>0$ be sufficiently small so that $a+\epsilon< t_1-\epsilon$ and define
\begin{equation*} 
    \theta_{{t_1},\epsilon}(t) = 
    \begin{cases}
        \frac{t-a}{\epsilon} & a \leq t < a + \epsilon, \\
        1 & a + \epsilon \le t <  {t_1}-\epsilon,\\
        \frac{{t_1}-t}{\epsilon} & {t_1}-\epsilon \le t < {t_1},\\
        0 & \text{otherwise}
    \end{cases}
\end{equation*}
to use in place of $\theta$. Take $\epsilon \rightarrow 0$ to obtain
    \begin{equation*} 
        \begin{aligned}
            \H[\eta, m|\bar{\eta}, \bar{m}] ({t_1}-) - \H[\eta, m| \bar{\eta}, \bar{m}] (a+) & \leq  \int_{a}^{t_1} \int_X \left(\frac{\bar{{v}}_x}{\rho_0}\right) \psi'(\tau | \bar{\tau}) \:\dr dt  \\
            &  + \int_{a}^{t_1} \iint_{X \times X} [\bar{{v}}](x, x'){K}_a'(\eta | \bar\eta)d\rho_0(x)d\rho_0(x') dt \\
            &  + \int_{a}^{t_1} \iint_{X \times X} [\bar{{v}}](x, x'){K}_r'(\eta | \bar\eta)d\rho_0(x)d\rho_0(x') dt.
        \end{aligned}
    \end{equation*}
Last, use $\eqref{dissip}_2$ and \eqref{jump} to obtain \eqref{relineq}.
\end{proof}
\begin{corollary} 
There exists $C_0>0$ such that for any ${t_1} \in I$
    \begin{equation} \label{e*}
        \begin{aligned}
            \H[\eta, {m} | \bar{\eta}, \bar{{m}}]({t_1}-) \leq \H[\eta, {m} | \bar{\eta}, \bar{{m}}](a) + C_0\int_{a}^{{t_1}} \mathcal{E}[\eta | \bar{\eta}]dt.
        \end{aligned}
    \end{equation}
\end{corollary}
\begin{proof}
 By Theorem \ref{ineqthm} and estimates \eqref{e1}-\eqref{e3}, taking $C_0 = \max\{C_1,C_2,C_3\}>0$ in \eqref{e1}-\eqref{e3} will give the result.
\end{proof}

\subsection{Applications of relative Hamiltonian inequality} \label{appl}

The relative Hamiltonian inequality provides a stability estimate thanks to the strict convexity via Theorems \ref{lowerbound}. We present two typical applications of the stability estimate. In Section \ref{appl}, we simply put $I = [a,b)=[0,T)$.
\subsubsection{Weak-strong uniqueness}

\begin{theorem}[Weak-strong uniqueness]
    Assume \eqref{hyp}. Suppose that $\eta(0,x) = \bar\eta(0,x)$, $\dot{\eta}(0,x) = \dot{\bar\eta}(0,x)$ for $\rho_0$-a.e. $x$. Then $\eta(t,x) = \bar\eta(t,x)$ for all $t\in I$ and $\rho_0$-a.e. $x$. 
\end{theorem}
\begin{proof}
 Assumptions on initial data imply that $\H[\eta,m | \bar\eta,\bar{m}](0)=0$. Since $\H[\eta,m|\bar\eta,\bar{m}]\ge \E[\eta|\bar\eta]$, estimate \eqref{e*} with $\H[\eta,m|\bar\eta,\bar{m}]$ in place of $\E[\eta|\bar\eta]$ in the integrand gives that $\H[\eta,m|\bar\eta,\bar{m}](t)=0$ for a.e. $t \in I$. This implies that for a.e. $t \in I$, $\E[\eta|\bar\eta](t)=0$ and that 
 $$\eta(t,x) = \bar\eta(t,x) + C(t) \quad \text{for $\rho_0$-a.e. $x$.}$$ Considering the regularity of $\eta$ with respect to time, the above equation holds for all $t\in I$ and $C(t)$ is absolutely continuous. By using $\K[\eta|\bar\eta](t)=0$ a.e., $C(t)\equiv 0$ and the assertion follows.
\end{proof}

\subsubsection{Uniform stability of a rarefaction}
We consider a motion $\bar\eta$, where $\dot{\bar{\eta}}(t,\cdot)$ is monotone increasing in $x$, thus undergoing rarefaction.
\begin{theorem}[Uniform stability] Assume \eqref{hyp}, $p \in (-1,1)\setminus\{0\}$, and $q\in (1,2)$. Assume also that for $t\in I$, $\bar{{v}}(t,\cdot)$ is monotone increasing. Let 
\begin{align*}
    \ell(t)&:=\min_{x>x'} \left\{ \frac{\bar{v}(x) - \bar{v}(x')}{x-x'}\right\}\ge 0,\\
    A(t) &:= \frac{\min\big\{\gamma+1, 2-p, 2-q\big\}\ell(t)}{\Lip(\eta(t,\cdot)) + \Lip(\bar\eta(t,\cdot))} \ge 0.
\end{align*}
Then we have the decay estimate 
\begin{align} \label{decay}
 \H[\eta,m|\bar\eta,\bar{m}](t-) + \int_{0}^t A(s)\E[\eta|\bar\eta](s) \:ds \le \H[\eta,m|\bar\eta,\bar{m}](0) 
\end{align}
\end{theorem}

\begin{proof}
Recall from Proposition \ref{pp} that 
\begin{equation*}
 \bar{v}_x\psi'(\tau| \bar\tau) = -\gamma(\gamma^2-1) (\tau - \bar\tau)^2 \int_0^1\int_0^s (\tau^\lambda)^{-\gamma-1}~\frac{\bar{v}_x}{\tau^\lambda} \: d\lambda ds.
\end{equation*}
From monotonicity of $\bar{v}$, $\eta$, and $\bar\eta$ we have that for each $t\in I$
\begin{align*}
\frac{\bar{v}_x}{\lambda\eta_x + (1-\lambda)\bar\eta_x} \ge \frac{\ell(t)}{\Lip(\eta(t,\cdot)) + \Lip(\bar\eta(t,\cdot))}
\end{align*}
We also have that
\begin{equation*}
\begin{aligned}
[\bar v] K_r'(\eta|\bar\eta)  &= -(p-1)(p-2)[\eta-\bar\eta]^2\int_0^1 \int_0^s \Big|[\eta^\lambda]\Big|^{p-2} \frac{[\bar{v}]}{[\eta^\lambda]} \:d\lambda ds,\\
[\bar v] K_a'(\eta|\bar\eta)  &= +(q-1)(q-2)[\eta-\bar\eta]^2\int_0^1 \int_0^s \Big|[\eta^\lambda]\Big|^{q-2} \frac{[\bar{v}}{[\eta^\lambda]} \:d\lambda ds,
\end{aligned}
\end{equation*}
and 
$$\frac{[\bar{v}]}{[\lambda \eta + (1-\lambda)\bar\eta]} \ge \frac{\ell(t)}{\Lip(\eta(t,\cdot)) + \Lip(\bar\eta(t,\cdot))}.$$
Since $\gamma>1$, $p\in(-1,1)$, and $q\in(1,2)$, $\min\big\{\gamma(\gamma^2-1), (p-1)(p-2), -(q-1)(q-2)\big\}$ is a positive constant. By \eqref{relineq}, we have the decay estimate \eqref{decay}.
\end{proof}
\begin{remark}
The function $A(t)$ is bounded from above by a constant multiple of $\displaystyle\frac{\ell(t)}{\Lip(\bar\eta(t))} \le \frac{\ell(t)}{ \ell(0) + \int_0^t \ell(s)\:ds}$. Note that $A(t)$ may approach zero as $t \rightarrow \infty$ depending on the behavior of $\ell(t)$. Therefore, the decay estimate \eqref{decay} does not necessarily imply that $\E[\eta|\bar\eta]$ decays to $0$ as $t \rightarrow \infty$. 
\end{remark}

\section{Relative Hamiltonian computation for presureless system in one space dimension} \label{specialEP}
We present further consequences for a system without presure. In this system, $\E_{w}$ is dropped and for $\E_r$ and $\E_q$, we choose $p=1$ and $q=2$ respectively. For this system in one space dimension, we make the following observations:
\begin{enumerate}
 \item The relative Hamiltonian measures $L^2$ distance of states:
 $$ \H[\eta,m|\bar\eta,\bar{m}] = \frac{1}{2} \|v - \bar{v}\|_{L^2_{\rho_0}}^2 + \frac{1}{2}\|\eta - \bar\eta_{a(t)}\|_{L^2_{\rho_0}}^2,$$
 where $\bar\eta_{a(t)}(t,\cdot): = \bar\eta(t,\cdot) - a(t)$ for each $t$, where the shift factor $a(t)$ is so that the centers of mass of $\eta(t,\cdot)$ and $\bar\eta(t,\cdot)$ coincide. In particular, we will observe that the relative energy contribution from the repulsive kernel $K_r(y) = -\frac{|y|}{4}$ simply vanishes.
 \item Under the assumption that $\eta$ and $\bar\eta$ are both strictly monotone, i.e., no  $\delta$-shock appears, we will observe that the $L^2$ distance is decreasing in time, i.e. we have uniform stability
 $$\int_X \frac{|v-\bar v|^2}{2} + \frac{|\eta - \bar\eta_{a(t)}|^2}{2} \:\dr\bigg|_{t} \le \int_X \frac{|v-\bar v|^2}{2} + \frac{|\eta - \bar\eta_{a(0)}|^2}{2} \:\bigg|_{t=0} \quad \text{for a.e. $t$.}$$
 Even for $\eta$ that is merely monotone, we have at most linear growth of the squared $L^2$ distance. These are much stronger stability estimates than those based on Gronwall type inequalities, the latter being with the constant that may grow exponentially in time.
\end{enumerate}

\subsection{Dissipative weak solution and strong solution}
 \label{sec:FP}
 Under the assumptions of the formulation (B) in Section \ref{sec:(B)}, the interaction energy is given by 
\begin{align*}
 &\E = \E_r + \E_a = \iint_{X\times X} K_r([\eta]) + K_a([\eta])\:\drr, \\
 &K_r(y) = - \frac{|y|}{4}, \quad K_a(y) = \frac{y^2}{4},\end{align*}
and the Hamiltonian $\H[\eta,m] = \K[m] + \E[\eta]$, where $m= \rho_0v$, $v=\dot{\eta}$.

The Hamiltonian is finite due to \eqref{assumption}. We also define the {\it virtual work} functional $\mathbf{f}[\eta]$ 
$$ \phi \mapsto -\frac{d}{ds} \E[\eta + s\phi] \Big|_{s=0}$$
on $\phi\in L^2_{\rho_0}(X)$. Continuity of $\mathbf{f}[\eta]$ is justified in Proposition \ref{vw2}. 
\begin{proposition} \label{vw2}
 Suppose $\E$ is given by the formula \eqref{energynormal}. Then the functional $\mathbf{f}[\eta] : \phi \mapsto -\frac{d}{ds} \E[\eta + s\phi] \Big|_{s=0}$ is continuous on
  $\phi \in L^2_{\rho_0}(X)$
 and has the representation
 \begin{equation} \label{f}
 \begin{aligned}
  \mathbf{f}[\eta](\phi)=& \iint_{X\times X} \Big(\frac{\sgn([\eta])}{4} - \frac{[\eta]}{2}\Big)[\phi]\: d\rho_0(x)d\rho_0(x').
 \end{aligned}
 \end{equation}
\end{proposition}
\begin{proof} By triangle inequality,
\begin{align*}
  &\iint_{X\times X}\left| \frac{1}{4s}\Big| (\eta + s\phi)(x) - (\eta + s\phi)(x')\Big|- \frac{1}{4s}\Big|\eta(x) - \eta(x')\Big| \right|\: d\rho_0(x)d\rho_0(x') \\
  &\le  \frac{1}{4}\iint_{X\times X}\Big|\phi(x) - \phi(x')\Big| \: \drr\\
  &\le \frac{1}{2} \int_X |\phi(x)|\sqrt{\rho_0(x)} \sqrt{\rho_0(x)} \:dx < \infty.
\end{align*}
Also, for $|s| \le 1$
\begin{align*}
 &\iint_{X\times X} \left|\frac{1}{4s}\Big( (\eta + s\phi)(x) - (\eta + s\phi)(x')\Big)^2- \frac{1}{4s}\Big(\eta(x) - \eta(x')\Big)^2\right| \: d\rho_0(x)d\rho_0(x') \\
  &= \frac{1}{4} \iint_{X\times X}\Big| (2\eta + s\phi)(x) - (2\eta + s\phi)(x')\Big|\Big|\phi(x)-\phi(x')\Big|\: d\rho_0(x)d\rho_0(x')\\
  &\le \iint_{X\times X} \frac{1}{2}\Big| \eta(x) - \eta(x')\Big|\Big|\phi(x)-\phi(x')\Big| + \frac{1}{4} \Big(\phi(x)-\phi(x')\Big)^2  \: d\rho_0(x)d\rho_0(x') \\
  & < \infty.
\end{align*}
Taking the limit as $s \rightarrow 0$, we obtain the result. Continuity can be readily inferred from the given representation.
\end{proof}

To give a definition of weak Hamiltonian flow, we proceed as in \cite{2013Brenier}. For an orbit $\eta \in \mathcal{A}(I\times X)$, the infinitesimal differential at $\eta(t,\cdot)$ must be restricted so that $\eta(t,\cdot)$ stays in the cone $\textsf{S}(X)$ of increasing maps. \cite{2013Brenier} investigated the normal and tangent cones at a given state $\eta$ to define the Lagrangian solution \cite[definition 3.4]{2013Brenier}. Explicit use of the normal and the tangent cone in a differential inclusion gives the precise description of the weak motion.

In this paper, the weak flow is defined  in a rather loose manner, that is not as precise as that in \cite{2013Brenier}. Nevertheless, following \cite[equation (1.33) and Lemma 2.4]{2013Brenier}, we define, for an increasing map $\eta \in L^2_{\rho_0}(X)$, the set:
\begin{align} \label{vacuumset}
    \Omega_{\eta}&:= \left\{x\in X ~|~ \eta \text{ is constant in a neighborhood of $x$}\right\},
\end{align}
which is a countable union of disjoint open intervals. Subsequently, we define the cone of admissible perturbations as
\begin{align*}
    T_{\eta}&:= \left\{w\in L^2_{\rho_0}(X) ~|~ w \text{ is increasing in every $(\alpha,\beta) \subset \Omega_{\eta}$}\right\}.
\end{align*}
Having defined the sets $\Omega_{\eta}$ and $T_{\eta}$, a map $\eta$ defined on $I\times X$ satisfying \eqref{assumption}, with $m=\rho_0 v$, $v= \dot{\eta}$,  is said to be a dissipative weak solution if 
\begin{align} \label{wksol2}
         -\int_I\int_{X}^{}{m}\dot{\phi}(t,x) \: dxdt &= \int_I \mathbf{f}{[\eta(t,\cdot)]}(\phi(t,\cdot))
\end{align}
holds for all $\phi \in C_c^{0,1}((a,b) \times X)$ satisfying $\dot\phi(t,\cdot) \in T_{\eta(t,\cdot)}$ for every $t\in I$, and it satisfies
    \begin{equation} \label{dissip2}
        \begin{aligned}
        &-\int_I \H[\eta, {m}]\dot{\theta}(t)dt \leq 0, \quad &&\text{for any non-negative $\theta\in C^{0,1}_c((a,b))$}, \\
        &\H[\eta,m](a+) \le \H[\eta,m](a).
        \end{aligned}
    \end{equation}
A dissipative weak solution is said to be a \textit{strong solution}, denoted by $\bar{\eta}$, if it additionally satisfies the following:
\begin{enumerate} [label=(\roman*)]
 \item For each $t\in I$, $\bar\eta(t,\cdot)$ is strictly increasing and $\ddot{\bar\eta}(t,\cdot)$ is in $L^\infty(X)$.
 \item 
 \begin{equation} \label{strsol2}
  \int_X \dot{\bar{{m}}} \phi(x) \: dx = \mathbf{f}[\bar\eta](\phi) \quad \text{for all $\phi\in C_c(X)$, and for a.e. $t\in I$.}
 \end{equation}
 \item The two inequalities in $\eqref{dissip2}$ are satisfied as equalities.
\end{enumerate}
Throughout the Section \ref{specialEP}, we fix $\H[\eta,m] = \K[m] + \E[\eta]$, where $\E[\eta]$ is given as in \eqref{energynormal}, $I=[a,b)$, and we assume
\begin{equation} \label{hyp2}
    \begin{aligned}
        &\text{$\eta$ is a dissipative weak  flow for $\H$ on $I$,}\\
        &\text{$\bar\eta$ is a strong flow for $\H$ on $I$,}\\
        &\text{${v}=\dot\eta$, $\bar{v}=\dot{\bar{\eta}}$, ${m}= \rho_0 {v}$, and $\bar{m}= \rho_0 \bar{v}$.} 
    \end{aligned}
\end{equation}

\subsection{Convexity and relative Hamiltonian as squared $L^2$ distance}
The relative energies  are defined similarly to \eqref{exp1}-\eqref{expK}. A simple calculation yields
\begin{align*}
 \E_a[\eta|\bar\eta] &= \frac{1}{4}\iint [\eta-\bar\eta]^2\:\drr,\\
 \E_r[\eta|\bar\eta] &= \frac{1}{4}\iint -|[\eta]| + |[\bar\eta]| +\sgn([\bar\eta])[\eta-\bar\eta] \:\drr\\
  &= \frac{1}{4}\iint [\eta]\Big(\sgn([\bar\eta]) - \sgn([\eta])\Big)\:\drr = 0 \\
  \K[m|\bar m]  &= \int_X \frac{\rho_0(v-\bar v)^2}{2} \: dx.
\end{align*}
Notably, the relative energy of the repulsive part vanishes. This was demonstrated in Theorem \ref{lowerbound}. 

Under the assumptions of \eqref{assumption} in Section \ref{sec:(B)}, the  center of mass is defined as
$$ y_{cm}[\eta](t) := \frac{\int_X \eta(t,x) \:\dr}{\int_X \dr}.$$
 Let $\bar\eta_{a(t)}(t,\cdot): = \bar\eta(t,\cdot) - a(t)$, where $a(t) = y_{cm}[\bar\eta(t,\cdot)] - y_{cm}[\eta(t,\cdot)]$. Then the centers of mass $y_{cm}[\bar\eta_{a(t)}]$ and $y_{cm}[\eta]$ coincide, i.e.,
$$ \int_X \big(\eta - \bar\eta_{a(t)}\big) \: \dr =0.$$
Consequently, the relative energy for the attractive kernel can be expressed as follows:
\begin{align*}
 &\E[\eta|\bar\eta] = \frac{1}{4}\iint [\eta-\bar\eta]^2\:\drr = \frac{1}{4}\iint [\eta-\bar\eta_{a(t)}]^2\:\drr\\
 &= \frac{1}{4} \iint_{X\times X} \big(\eta(x)-\bar\eta_{a(t)}(x)\big)^2 + \big(\eta(x')-\bar\eta_{a(t)}(x')\big)^2 \\
 & \quad \quad \quad - 2\big(\eta(x)-\bar\eta_{a(t)}(x)\big)\big(\eta(x')-\bar\eta_{a(t)}(x')\big)\:\drr\\
 &= \frac{1}{2}\int_X \big(\eta(x)-\bar\eta_{a(t)}\big)^2 \:\dr.
\end{align*}
In conclusion, the relative Hamiltonian $\H[\eta,m|\bar\eta,\bar{m}] = \K[m|\bar{m}] + \E[\eta|\bar\eta]$ measures $L^2$ distance of weak and strong states after the latter being shifted to match the centers of mass. More specifically, we have an exact equality 
$$ \H[\eta,m|\bar\eta,\bar{m}] = \frac{1}{2} \|v - \bar{v}\|_{L^2_{\rho_0}}^2 + \frac{1}{2}\|\eta - \bar\eta_{a(t)}\|_{L^2_{\rho_0}}^2.$$

\subsection{Relative Hamiltonian inequality}
Now, we justify the relative Hamiltonian inequality. In the rest of Section \ref{specialEP}, we put $I=[0,T).$
\begin{theorem} \label{ineqthm2}
    Assume \eqref{hyp2}. Then the map $t \mapsto \H[\eta, m|\bar{\eta}, \bar{m}](t)$ is a function of Bounded Variation on $I$, and for eath $t_1\in I$ it satisfies the following:
 \begin{enumerate} [label=(\roman*)]
  \item If $\eta(t,\cdot)$ is strictly monotone for each $0\le t\le t_1$, then
    \begin{equation} \label{relineq2}
            \H[\eta, m|\bar{\eta}, \bar{m}](t_1-) - \H[\eta, m| \bar{\eta}, \bar{m}](0) \le 0.
    \end{equation}
  \item Otherwise, we have 
    \begin{equation} \label{relineq3}
    \begin{aligned}
            &\H[\eta, m|\bar{\eta}, \bar{m}](t_1-) - \H[\eta, m| \bar{\eta}, \bar{m}](0) \\
            & \leq  \frac{1}{2}\int_0^{t_1}\int_{\{(x,x')~|~\eta(t,x)=\eta(t,x'), ~x>x'\}} \big(\bar{v}(t, x) - \bar{v}(t, x')\big)\: \drr dt.
    \end{aligned}
    \end{equation}
  \end{enumerate}
\end{theorem}

\begin{proof}
Similarly to \eqref{firsteqn},
\begin{align*}
        -\int_I &\H[\eta, {m} | \bar{\eta}, {\bar{m}}] \dot\theta(t)\:dt \\
        & \le \int_I\iint_{X \times X} \Big(-\frac{\sgn([\bar\eta])}{4} + \frac{[\bar\eta]}{2}\Big)[\eta-\bar{\eta}]\dot\theta \: \drr dt\\
        &  + \int_I\int_X \bar{v}({m}-\bar{{m}})\dot\theta dxdt.
\end{align*}
Since $\bar\eta$ is a strong solution, 
\begin{align*}
    \Big(-\frac{\sgn([\bar\eta])}{4} + \frac{[\bar\eta]}{2}\Big)[\eta-\bar{\eta}]\dot\theta&=\left\{-\frac{\sgn(x-x')\dot{\theta}}{4} + \frac{\big(\theta[\bar\eta]\big)^\bdot-\theta[\bar v]}{2}\right\}[\eta - \bar\eta], \\
    ({m}-\bar{{m}})\bar{v}\dot\theta&= 
    \Big(\big(\theta \bar{v}\big)^\bdot - \theta \dot{\bar{v}}\Big)({m}-\bar{{m}}).
\end{align*}
Similarly to \eqref{secondeqn},
\begin{align*}
 &\int_I\int_X -\theta \dot{\bar{v}}(m-\bar m) \:dx dt \\
 +&\int_I\iint_{X\times X} \Big(-\frac{\theta \sgn(x-x')}{4} + \frac{\theta[\bar\eta]}{2} \Big)^\bdot [\eta-\bar\eta] \:\drr dt = 0.
\end{align*}
Remaining terms are
\begin{align*}
 &\int_I\int_X \big(\theta \bar{v}\big)^\bdot ({m}-\bar{{m}}) \:dxdt - \frac{1}{2}\int_I\iint_{X\times X}\theta [\bar v][\eta-\bar\eta] \:\drr dt\\
 &= \int_I \theta \int_{X\times X} [\bar v] \Big( K'([\eta]) - K'([\bar\eta]) - \frac{[\eta-\bar\eta]}{2}\Big) \: \drr dt\\
 &= \int_I \theta \int_{X\times X} [\bar v] \Big( -\frac{\sgn([\eta])}{4} + \frac{ \sgn([\bar\eta])}{4}\Big)\: \drr dt\\
  &= \frac{1}{2}\int_I \theta \int_{\{\eta(x)=\eta(x'), x>x'\}} \big(\bar{v}(x) - \bar{v}(x')\big)\: \drr dt.
\end{align*}
Simlarly as in the proof of Theorem \ref{ineqthm}, we obtain the result.
\end{proof}

\subsection{Applications of relative Hamiltonian inequality}
\subsubsection{Uniform stability before delta shock formation}
\begin{theorem} \label{thm:L2stab1}
 Assume \eqref{hyp2}, and let $\bar\eta_{a(t)}$ be the translation $\bar\eta(t,\cdot) -a(t)$, $a(t) = y_{cm}[\bar\eta(t,\cdot)] - y_{cm}[\eta(t,\cdot)]$ for $t\in [0,T)$. Suppose further that for each $t\in [0,T)$ $\eta(t,\cdot)$ is strictly monotone. Then for $t\in[0,T)$
 $$\int_X \frac{|v-\bar v|^2}{2} + \frac{|\eta - \bar\eta_{a(t)}|^2}{2} \:\dr\bigg|_{t-} \le \int_X \frac{|v-\bar v|^2}{2} + \frac{|\eta - \bar\eta_{a(0)}|^2}{2} \:\bigg|_{t=0}.$$
\end{theorem}
\begin{proof}
    The estimate \eqref{relineq2} furnishes the proof.
\end{proof}
Theorem \ref{thm:L2stab1} tells that, as long as one solution is strong and both of $\eta$ and $\bar \eta$ are strictly monotone (i.e., before the formation of a $\delta$-shock), the squared $L^2$ distance of two solutions does not increase over time. This $L^2$ stability is much stronger than what is expected from the typical practice of relative energy estimates. The next theorem provides a stability result even after the formation of a $\delta$-shock.

\subsubsection{Linear growth of the squared $L^2$ distance after delta shock formation}
\begin{theorem} \label{thm:L2stab2}
Assume \eqref{hyp2}, and let $\bar\eta_{a(t)}$ be the translation $\bar\eta(t,\cdot) -a(t)$, $a(t) = y_{cm}[\bar\eta(t,\cdot)] - y_{cm}[\eta(t,\cdot)]$ for $t\in [0,T)$. Then for $t\in[0,T)$ 
 $$\int_X \frac{|v-\bar v|^2}{2} + \frac{|\eta - \bar\eta_{a(t)}|^2}{2} \:\dr\bigg|_{t-} \le \int_X \frac{|v-\bar v|^2}{2} + \frac{|\eta - \bar\eta_{c(0)}|^2}{2} \:\bigg|_{t=0} + C_0t.$$ 
\end{theorem}
\begin{proof}
    The estimate \eqref{relineq3} furnishes the proof. 
\end{proof}
\subsubsection{Large friction limit}
In this section, we study the damped Euler-Poisson system
\begin{equation}\label{dEP}
 \begin{split}
  \dot{\eta}_\eps &= \frac{\delta \H}{\delta m}[\eta_\eps,m_\eps], \\
  {\dot{ m}}_\eps &= -\frac{\delta \H}{\delta \eta}[\eta_\eps,m_\eps] - \frac{1}{\eps} m^\eps.
 \end{split}
\end{equation}
In the second equation of \eqref{dEP}, the presence of a large friction term, parameterized by a small constant $\eps > 0$, suggests that the body under consideration is expected to rapidly lose momentum. This leads to a phase where the momentum remains small, resulting in comparable effects between the two terms on the right-hand side of equation $\eqref{dEP}_2$. Our focus lies on understanding this intermediate asymptotic behavior, which occurs before the body ceases to move as $t \rightarrow \infty$. We aim to compare this motion to that governed by the gradient flow \eqref{GF}. %

We adopt the methodology presented in \cite{2017Lattanzio}. To observe the aforementioned effects while the momentum remains small, we introduce a coarse time scale $s$ so that $t = \frac{s}{\eps}$. We use the notation $f'(s,x) = \partial_s f(s,x)$ and define the following scaled variables. 
\begin{align*}
 \hat\eta_{\eps}(s,\cdot)=\eta_\eps\left(\frac{s}{\eps},\cdot\right), \quad \hat{v}_{\eps}(s,x)&=\hat\eta_\eps'(s,x) = \frac{1}{\eps}\dot{\eta}_\eps\left(\frac{s}{\eps},x\right),\\
 \hat{m}_{\eps}(s,x) &=\rho_0(x)\hat{v}_\eps(s,x)= \frac{1}{\eps}  m_\eps\left(\frac{s}{\eps},x\right).
\end{align*}
For convenience, we do not explicitly write the dependence of $\eps$ in the variables $\eta_\eps$ and $m_\eps$, whenever it is clear from the context. We can express the Hamiltonian $\H[\eta,m](t)$ with respect to the scaled variables $\hat{\eta}(s, \cdot)$ and $\hat{m}(s, \cdot)$ where $t = s/\eps$:
\begin{align*}
 \H[\eta,m](t) = \hat\H[\hat\eta,\hat m](s) := \eps^2\int_X \frac{|\hat{m}(s, \cdot)|^2}{2\rho_0}\:dx + \E[\hat\eta](s).
\end{align*}
Then \eqref{dEP} leads to the scaled system of equations
\begin{equation}\label{eq:hamsca}
 \begin{split}
  \eps^2\hat{\eta}' &= ~~\frac{\delta \hat\H}{\delta \hat{m}}[\hat{\eta}, \hat{m}] = \eps^2 \frac{\hat{m}}{\rho_0}, \\
  \eps^2\hat{m}' + \hat{m} &= -\frac{\delta \hat\H}{\delta \hat{\eta}}[\hat{\eta}, \hat{m}] = -\frac{\delta \E}{\delta \hat\eta} [\hat{\eta}].
 \end{split}
\end{equation}
The formal limit of system \eqref{eq:hamsca} at $\eps = 0$ are the two equalities
$$ {\bar\eta}' = \frac{\bar m}{\rho_0}, \quad \bar m= - \frac{\delta \E}{\delta \hat\eta}[\bar\eta],$$
or 
\begin{equation} \label{GF}
\rho_0 \bar\eta' = - \frac{\delta \E}{\delta \hat\eta}[\bar\eta]
\end{equation}
the gradient flow of $\E$. 

The systems \eqref{dEP} and \eqref{GF} have already been discussed in \cite[Remark 4.3]{2016Carrilloa}. In particular, the local well-posedness for smooth solutions of \eqref{dEP} is proven in \cite[Appendix A]{2016Carrilloa}. If $K_a$ is absent, or the interaction is purely repulsive and the kernel consists solely of the Newtonian part, well-posedness of \eqref{GF} has been studied by \cite{2015Bonaschi} in one dimension and \cite{2012Bertozzi} in higher dimensions. In their studies \cite{2015Bonaschi,2012Bertozzi}, they also present results for the sign corresponding to attractive interactions.

The purpose of the remaining part of this section is to provide a rigorous proof of the continuity of the flows with respect to $\eps$ at $0$ under our regularity assumptions. Specifically, this continuity is established under the assumption that there exists a family of dissipative weak solutions $(\eta_\eps, m_\eps)_{\eps>0}$ to \eqref{eq:hamsca} and a strong solution to the gradient flow \eqref{GF}. 

Regularity assumptions on the dissipative weak solution and the strong solution are precisely those stated in Section \ref{sec:FP}, along with the following additional conditions. The continuity at $\eps=0$ is proven under the absence of $\delta$-shock. Specifically, we assume that:
\begin{equation}\label{additional}
    \begin{aligned}
     &\text{$\eta(s,\cdot)$ and $\bar\eta(s,\cdot)$ are strictly monotone for each $s \in I$.}
    \end{aligned}
 \end{equation}
The existence result in \cite{2016Carrilloa} demonstrates that the solution locally persists and remains strictly monotone in time. However, it is important to note that the assumption of strict monotonicity throughout the lifetime of the solutions for \eqref{dEP} and those for the gradient flow \eqref{GF} is included as an assumption in our current analysis.

One notational remark is that we have introduced $(\bar\eta,\bar{m})(s,\cdot)$ for the strong solution of \eqref{GF}, which, in this section, is already expressed in terms of $s$. Also, the space $L^2_{\rho_0}$ is equipped with the scalar product
\[ (v,w) \mapsto \int_{X} v(x)w(x)\:\dr, \]
and $-\frac{\delta \E}{\delta \hat{\eta}}$ is a bounded linear functional on $L^2_{\rho_0}$ (see Proposition \ref{vw2}). Thus the expression $- \frac{1}{\rho_0}\frac{\delta \E}{\delta \hat\eta}[\hat\eta]$ represents an element of $L^2_{\rho_0}$ by the Riesz representation.

For the dissipative weak solution, the equations and inequalities \eqref{wksol2}-\eqref{dissip2} are replaced, respectively, as follows:
\begin{align}
&\begin{aligned} 
&-\int_I\int_{X} \eps^2{\hat m}(s,x){\phi}'(s,x) + \hat{m}(s,x)\phi(s,x)\: dxds \\
&= \int_I \mathbf{f}{[\hat \eta(s,\cdot)]}\big(\phi(s,\cdot)\big) \: ds, \quad \text{for all $\phi \in C_c^{0,1}(I \times X)$,} \label{wksol3}
\end{aligned} \\ 
&\left\{\begin{aligned}
    &-\int_I \hat\H[\hat\eta, \hat{m}]{\theta}'(s)ds + \int_I \theta(s)\int_X \hat{m}(s,x)\hat{v}(s,x) \:dxds \leq 0, \\
    &\text{for any non-negative $\theta\in C^{0,1}_c((a,b))$ and} \\
    &\hat\H[\eta,m](a+) \le \hat\H[\eta,m](a).\label{dissip3}
\end{aligned}  \right.
\end{align}

As to the strong solution, \eqref{strsol2} is replaced by
\begin{align*}
    \int_X {\bar{{m}}} \phi(x) \: dx = \mathbf{f}[\bar\eta](\phi(\cdot))
\end{align*}
for all $\phi\in C_c(X)$, and for a.e. $t\in I$. With the expression $R(s,x) := \eps^2\bar{m}'(s,x)$, we will use for strong solution:
\begin{align} \label{strsol3-2}
 &\eps^2\int_X \bar{m}'\phi + \int_X (\bar{m}-R) \phi(x) \: dx = \mathbf{f}[\bar\eta](\phi),\\
  &-\int_I \hat\H[\bar\eta, \bar{m}]{\theta}'(s)ds + \int_I \theta(s)\int_X \big(\bar{m}(s,x)- R(s,x)\big)\bar{v}(s,x) \:dxds = 0. \label{dissip3-2}
 \end{align}

The proximity between the solution of \eqref{dEP} and that of \eqref{GF} is proven, assuming the proximity between the initial states in the following sense.
 \begin{equation}
 \label{eq:condini}
 \begin{aligned}
    \| \hat\eta_{\eps}(0,\cdot) - \bar\eta_{a(t)}(0,\cdot)\|_{L^2_{\rho_0}} = \mathcal{O}(\eps^2), \quad \Big\|\hat{\eta}_\eps'(0,\cdot) -\Big(-\frac{1}{\rho_0} \frac{\delta \E}{\delta \hat\eta}(\bar\eta(0,\cdot)\Big)\Big\|_{L^2_{\rho_0}} =\mathcal{O}(\eps). 
 \end{aligned}
 \end{equation}

\begin{theorem} \label{thm:frictionlimit}
 Let $I=[0,T)$ for any $T>0$ and let $\bar\eta$ be a strong solution to \eqref{GF} and $(\hat\eta_\eps)_{\eps>0}$ be a family of (dissipative) solutions to \eqref{eq:hamsca} satisfying \eqref{additional} and initially \eqref{eq:condini}. Then  
 \[ \| \hat\eta_\eps(s,\cdot) - \bar\eta_{a(t)}(s,\cdot)\|_{L^2_{\rho_0}}=\mathcal{O}(\eps^2) \quad \text{and} \quad  \| \hat{\eta}_\eps'(s,\cdot) \|_{L^2_{\rho_0}}=\mathcal{O}(1)\]
 uniformly for $0 \leq s \leq T.$
\end{theorem}

\begin{remark}
    That $\|\hat\eta_\eps^\prime(s,\cdot)\|_{L^2_{\rho_0}} = \mathcal{O}(1)$ implies $\|\dot{\eta_\eps}(t,\cdot)\|_{L^2_{\rho_0}} = \mathcal{O}(\eps)$ at $t= \frac{s}{\eps}$.
   \end{remark}

\begin{proof}
Using \eqref{wksol3}, \eqref{dissip3}, \eqref{strsol3-2}, and \eqref{dissip3-2}, we have
\begin{align*}
        &-\int_I \hat\H[\hat\eta, \hat{m} | \bar{\eta}, {\bar{m}}] \theta'(s)\:ds \\
        & + \int_I\iint_{X \times X} \Big(-\frac{\sgn([\bar\eta])}{4} + \frac{[\bar\eta]}{2}\Big)[\eta-\bar{\eta}]\theta' \: \drr ds\\
        & \le \int_I \theta \int_X \big( -\hat{m}\hat{v} + (\bar{m} - R)\bar{v}\big)\:dxds + \eps^2\int_I\int_X (\hat{m}-\bar{{m}})\bar{v}\theta' dxds.
\end{align*}
Using equations \eqref{wksol3} and \eqref{strsol3-2} simlarly as in the proof of Theorem \ref{ineqthm2}, the right-hand-side equals 
\begin{equation} \label{residual}
\begin{aligned}
    & \int_I \theta \int_X  \bar{v}\big(\hat{m} -\bar{m} +R\big) + (\bar{m}-R)(\hat{v}-\bar{v}) + \big( -\hat{m}\hat{v} + (\bar{m} - R)\bar{v}\big) \:dxds\\
     +~~ \frac{1}{2}&\int_I \theta \int_{\{\eta(x)=\eta(x'), x>x'\}} \big(\bar{v}(x) - \bar{v}(x')\big)\: \drr ds.
\end{aligned}
\end{equation}
Under the assumption of strict monotonicity \eqref{additional}, the second integral of \eqref{residual} vanishes. The first integral of \eqref{residual} equals 
\begin{align*}
    =&~\int_I \theta \int_X  -\rho_0(\hat{v}-\bar{v})^2 + R(\bar{v}-\hat{v}\big) \:dxds \le \eps^4\int_I \theta \int_X  \rho_0(\bar{v}')^2 \:dxds.
\end{align*}
We conclude that 
\begin{align}\label{eq:reg} 
 \hat\H[\hat{\eta},\hat{m}|\bar\eta,\bar{m}](T) - \hat\H[\hat{\eta},\hat{m},\bar\eta,\bar{m}](0) \le \eps^4 T\|\bar{v}'\|_\infty^2.
\end{align}
 By the formula at each time $s \in I$
 $$\hat\H[\hat{\eta},\hat{m}|\bar\eta,\bar m](s) = \frac{\eps^2}{2}\|\hat{v}-\bar{v}\|^2_{L^2_{\rho_0}}(s) + \frac{1}{2} \|\hat\eta - \bar\eta_{a(s)}\|^2_{L^2_{\rho_0}}(s).$$
The assumptions in \eqref{eq:condini} ensure $\hat\H[\hat\eta_\eps, \hat{m}_\eps| \bar{\eta}, \bar{m }]\big|_{s=0} =\mathcal{O}(\eps^4)$, and
 \eqref{eq:reg} ensures $\hat\H[\hat\eta_\eps, \hat{m}_\eps| \bar{\eta}, \bar{m} ](s)= \mathcal{O}(\eps^4)$ for each finite time $s>0$. This establishes the first assertion. 
For the second assertion, it suffices to establish $\|\bar{v}\|_{L^2_{\rho_0}} = \mathcal{O}(1)$. By testing $\bar\eta$ on the gradient flow \eqref{GF}, it is straightforward to see 
$$ \frac{d}{ds}\|\bar\eta\|^2_{L^2_{\rho_0}} \le C \|\bar\eta\|^2_{L^2_{\rho_0}} \quad \text{for some $C>0$.}$$
Therefore $\|\bar\eta\|_{L^2_{\rho_0}}$ remains $\mathcal{O}(1)$ for each finite time $s>0$. Considering the form of \eqref{f} in Proposition \ref{vw2} and that $\rho_0\bar v = -\frac{\delta\E}{\delta \hat\eta}[\bar\eta]$, by Riesz representation $\|\bar{v}\|_{L^2_{\rho_0}} = \mathcal{O}(1)$.
\end{proof}

\begin{remark}
In fact, under the gradient flow \eqref{GF}, the center of mass of $\bar{\eta}$ remains stationary, which can be demonstrated as follows:
\begin{equation*}
    \frac{d}{ds} \int_{X} \bar\eta(s,x)\: \dr = \int_{X} \bar\eta'(s,x)\: \dr 
    = -\iint_{X\times X} K'([\bar\eta])\: \drr  =  0,
\end{equation*}
where we used the fact that $K'(y) = -\frac{\sgn(y)}{4} + \frac{y}{2}$ is an odd function. On the other hand, as shown in \cite[equation $(2.3)$]{2016Carrilloa}, the average velocity of the damped Hamiltonian solution decreases exponentially with a decay rate of $\eps^{-2}$. Consequently, the center of mass of $\hat{\eta}_\eps$ moves by, at most, $\mathcal{O}(\eps^2)$ over finite time. 
This implies that if the initial center of mass of $\hat{\eta}_\eps$ and that of $\bar{\eta}$ are distanced by $\mathcal{O}(\eps^2)$, then we have for any finite $s$
\[ \| \hat\eta_\eps(s,\cdot) - \bar\eta(s,\cdot)\|_{L^2_{\rho_0}}=\mathcal{O}(\eps^2) \]
without the shift factor.

\end{remark}

\appendix

\section{Appendix A}
\begin{appendixlemma} \label{lem:const}
 Suppose $\eta\in \hat{\textsf{S}}(X)$ is admissible. There exist positive constants $E_1$, $E_2$, and $E_3$ such that 
\begin{equation*}
\begin{aligned}
 &\iint_{X \times X} |x-x'|^{p-1}\Big| \big(\eta(x)-\eta(x')\big) -\big(\eta_*(x)-\eta_*(x')\big)\Big| d\rho_0(x)d\rho_0(x') < E_1,\\
 &\iint_{X \times X} |x-x'|^{q-1}\Big| \big(\eta(x)-\eta(x')\big) -\big(\eta_*(x)-\eta_*(x')\big)\Big| d\rho_0(x)d\rho_0(x') < E_2,\\
&\int_X \left|\frac{\eta_x(x) - \eta_{*x}(x)}{\rho_0}\right| \: \dr < E_3.
\end{aligned}
\end{equation*}
$E_1=E_1(\Lip(\eta),p)$, $E_2 = E_2(\Lip(\eta^{-1}),q)$, and $E_3=E_3(\Lip(\eta),\gamma)$.
\end{appendixlemma}
\begin{proof}
 For the first assertion, the case $p=1$ is trivial, and we take $p\in (-1,1)\setminus \{0\}$. For the second assertion we take $q>1$ with the same reason. We introduce $h(x,x')$ and $\bar h(x,x')$ as in the proof of Theorem \ref{lowerbound}.
For the integrand for $p\in (-1,1)\setminus \{0\}$
 \begin{align*}
  \Big| \frac{1}{p}\big|[\eta]\big|^p - \frac{1}{p}\big|[\eta_*]\big|^p\Big|
  &= \Big| \frac{h^p}{p} - \frac{h_*^p}{p}\Big| \, |x-x'|^p\\
  &= \Big| \int_0^1 (\lambda h + (1-\lambda)h_*)^{p-1} \:d\lambda\Big|\, |h-h_*|\,|x-x'|^p\\
  &\ge \max\{h,h_*\}^{p-1} \big|[\eta - \eta_*]\big| \,|x-x'|^{p-1}\\
  &\ge \max\{\Lip(\eta),\Lip(\eta_*)\}^{p-1}\big|[\eta - \eta_*]\big| \,|x-x'|^{p-1}.
 \end{align*}
For $q>1$,
 \begin{align*}
  \Big| \frac{1}{q}\big|[\eta]\big|^q - \frac{1}{q}\big|[\eta_*]\big|^q\Big|
  &\ge \min\{h,h_*\}^{q-1} \big|[\eta - \eta_*]\big| \,|x-x'|^{q-1}\\
  &\ge \min\left\{\frac{1}{\Lip(\eta^{-1})},\frac{1}{\Lip(\eta^{-1})}\right\}^{q-1} \big|[\eta - \eta_*]\big| \,|x-x'|^{q-1}\\
  &= \max\{\Lip(\eta^{-1}), \Lip(\eta_*^{-1})\}^{1-q} \big|[\eta - \eta_*]\big| \,|x-x'|^{q-1}
 \end{align*}
Therefore
 \begin{align*}
  &\iint_{X\times X} |x-x'|^{p-1}\, \big|[\eta - \eta_*]\big|\: \drr\\
  &\le\max\{\Lip(\eta),\Lip(\eta_*)\}^{1-p} \iint_{X\times X}   \Big| \frac{1}{p}\big|[\eta]\big|^p - \frac{1}{p}\big|[\eta_*]\big|^p\Big| \:\drr < \infty,\\
  &\iint_{X\times X} |x-x'|^{q-1}\, \big|[\eta - \eta_*]\big|\: \drr\\
  &\le\max\{\Lip(\eta^{-1}), \Lip(\eta_*^{-1})\}^{q-1} \iint_{X\times X} \Big| \frac{1}{q}\big|[\eta]\big|^q - \frac{1}{q}\big|[\eta_*]\big|^q\Big| \:\drr < \infty.
 \end{align*}
 
 For the third assetion, we have the integrand for $\rho_0$-a.e. $x$
 \begin{align*}
  |\tau^{1-\gamma} - \tau_*^{1-\gamma}| & = |1-\gamma| \Big|\int_0^1 (\lambda \tau + (1-\lambda)\tau_*)^{-\gamma} \:d\lambda\Big|\, |\tau-\tau_*|\\
  &\ge|1-\gamma|\max\{\Lip(\eta),\Lip(\eta_*)\}^{-\gamma} |\tau-\tau_*|, \quad \text{which gives}
  \end{align*}
  \begin{align*}
  \int_X |\tau - \tau_*| \: \dr \le \frac{\max\{\Lip(\eta),\Lip(\eta_*)\}^{\gamma}}{\gamma-1} \int_X |\tau^{1-\gamma} - \tau_*^{1-\gamma}| \:\dr < \infty.
 \end{align*}

 \end{proof}

 \section*{Declarations}

 \paragraph*{Funding} 
 J.G. is grateful for financial support by the German Science Foundation 
 (DFG) via grant TRR 154 (\emph{Mathematical modelling, simulation and 
 optimization using the example of gas networks}), project C05. 
 K. Kwon and M. Lee was supported by the National Research Foundation of Korea (NRF) grant funded by the Korea government (MSIT) (No. 2020R1A4A1018190, 2021R1C1C1011867).
 
 \paragraph*{Competing interests}
 The authors have no relevant financial or non-financial interests to disclose.

 \bibliographystyle{amsplain}
 \bibliography{bibliography/euler_poisson_refs}

\end{document}